\documentclass[12pt,reqno]{amsart}
\usepackage{geometry}
\usepackage{mathrsfs}
\usepackage{graphicx}
\usepackage{hyperref}
\usepackage{enumerate,diffcoeff}
\usepackage{latexsym,amsmath,amssymb,amsfonts}
\usepackage{color,colordvi,amscd,indentfirst,graphics}
\DeclareMathOperator{\sign}{sgn}
\allowdisplaybreaks

\newtheorem{theorem}{Theorem}[section]

\newtheorem{proposition}[theorem]{Proposition}

\newtheorem{claim}[theorem]{Proposition}
\newtheorem{corollary}[theorem]{Corollary}

\newtheorem{remark}{Remark}

\renewcommand{\Re}{\operatorname{Re}}
\renewcommand{\Im}{\operatorname{Im}}
\DeclareMathOperator{\Aut}{Aut}


\makeatletter
\@namedef{subjclassname@2020}{\textup{2022} Mathematics Subject Classification}
\makeatother
\thanks{The third-named author was supported by Vietnam National Foundation for Science and Technology Development under grant number IZVSZ2\_229554 (NAFOSTED -- SNSF Joint Research Project 2025)}
\date{October 14, 2025}

\begin{document}
\title[The classification of CR maps from hyperquadrics]{The classification of CR maps from hyperquadrics into 
tubes over null cones of symmetric forms}
\begin{abstract}
We classify CR maps from the hyperquadric of signature \(l>0\) in \(\mathbb{C}^n\), $n\geq 3$,
to the local model for the tube over
the null cone of a symmetric form in \(\mathbb{C}^{n+1}\), up to CR automorphisms
of the source and target. In contrast to the setting of the Heisenberg
hypersurface in \(\mathbb{C}^3\) (i.e., the case \(l=0\)),
studied earlier in Reiter--Son \cite{ReiSon2024}, our analysis uncovers two new equivalence classes of CR maps of geometric rank one and one new 
class of geometric rank two in the case $n=3$. In the case $n\geq 4$, we establish that all maps extend to local isometries of certain indefinite K\"ahler metrics. We further derive
a classification of (local) proper holomorphic maps from the generalized unit ball
\(\mathbb{B}^n_l\) into a generalized version of the Lie ball \(D^{\mathrm{IV}}_{m,l}\) (the generalized classical domain of type~IV).
\end{abstract}
\author{Nguyen Gia Hien}
\address{Faculty of Mathematics, Hanoi National University of Education, Hanoi 11310, Vietnam}
\email{hienng@hnue.edu.vn, hiennguyengia450@gmail.com}
\author{Michael Reiter}
\address{Fakultät für Mathematik, Universität Wien, Oskar-Morgenstern-Platz 1, 1090 Wien, Austria}
\email{m.reiter@univie.ac.at}
\author{Duong Ngoc Son}
\address{Faculty of Fundamental Sciences, Phenikaa University, Hanoi 12116, Vietnam}
\email{son.duongngoc@phenikaa-uni.edu.vn}

\maketitle

\section{Introduction}
Our primary interest of this paper is the characterization of local CR maps from the hyperquadric
of signature \(l>0\) in \(\mathbb{C}^n\) to the tube over the null cone of
a symmetric form in \(\mathbb{R}^{n+1}\) which is homogeneous and Levi-degenerate of signature \((l, n-l-1, 1)\) in
\(\mathbb{C}^{n+1}\). This problem
is closely related to the classification of local proper holomorphic mappings
from the generalized ball into the domain the ``generalized classical domain of type IV'' (or the generalized Lie ball), which is 
ultimately motivated by the classical works of Henri Poincaré \cite{Poi1907} and Alexander \cite{Ale1977}.
These two works focus on the classification of proper holomorphic maps on 
the complex unit balls of the same dimension. The different dimensional case was
studied in various later works. In 1979, Webster \cite{Web1979} 
provided a rigidity result on the proper holomorphic mappings from $\mathbb{B}^n$ to 
$\mathbb{B}^{n+1}$ which extend sufficiently smooth to a boundary point for $n\geq 3$. In 1982, Faran \cite{Far1982} completed the
classification problem for the codimension one case by providing a classication
of maps from $2$-ball to $3$-ball as four equivalence classes of maps. 
Later in 1986, Faran \cite{Far1986} proved that the
proper holomorphic mappings which extend holomorphically from $\mathbb{B}^n$ to
$\mathbb{B}^k$ with $n\geq 3, k\leq 2n-2$ are linear fractional. The \(C^2\)-smoothness case was treated in
Huang \cite{hu99}. For more
results on proper holomorphic mapping between balls, the readers may
refer to D'Angelo \cite{DAng1988}, Forstneri\v{c} \cite{Fors1989}, 
Della Sala et al \cite{DelSal2020}, and well as Baouendi--Huang \cite{BaoHua2005},
Huang et al \cite{HLTX2020} for the case of hyperquadrics,
and many references therein.

The problem to understand proper holomorphic mapping from the unit ball to the 
classical domains, or more generally, proper holomorphic mappings 
and holomorphic isometries between various type of classical domains has been
studied extensively, see the beautiful survey of Mok \cite{Mok2018} and the
references therein. In the special case of balls and type IV domains,
Xiao--Yuan \cite{XiaoYuan2020} proved in 2020 the rigidity of 
proper holomorphic maps from the complex unit ball $\mathbb{B}^n$ to the type
IV bounded symmetric domain $D^{\mathrm{IV}}_m$ (the $m$-dimensional Lie ball)
with $n\geq 4, n+1\leq m\leq 2n-3$, giving an explicit formula for the ``nonstandard'' isometry of Mok,
see also Xiao \cite{xiao23}.  
The lower dimensional cases, namely the cases $n=2$ and $n=3$, were settled recently
by Reiter--Son \cite{ReiSon2022} and \cite{ReiSon2024}, respectively. The aim of this paper is to extend
these works to obtain a classification of CR maps from the hyperquadric  
$\mathbb{H}^{2n-1}_l$ of signature \(l\) in \(\mathbb{C}^n\) into the tube over
the null cone \(\mathcal{X}^{2n+1}_l\)
of signature \((l, n-l-1, 1)\) (the numbers of negative, positive, and zero
eigenvalues of the Levi form). As an immediate consequence, we classify local proper
holomorphic maps from the generalized ball into the generalized Lie ball $D^{\mathrm{IV}}_{m,l}$ 
(or the generalized type IV domain).

We recall the definitions of indefinite product, the generalized
classical
domains with signature $l>0$, namely the generalized ball $\mathbb{B}^n_l$ and
the generalized Lie ball $D^{\mathrm{IV}}_{m,l}$, as well as the hyperquadric
$\mathbb{H}^{2n-1}_l$ and the local model $\mathcal{X}^{2n+1}_l$ for the tube over the null
cone of a symmetric form. Precisely, for 
$a=(a_1,\dots,a_{n-1})$ and $b=(b_1,\dots ,b_{n-1})\in\mathbb{C}^{n-1}$, we define the 
indefinite product $\langle a, b\rangle_l$ of signature \(l\) to be:
\begin{align*}
  \langle a, b\rangle_l=-\sum_{j=1}^la_jb_j+\sum_{k=l+1}^{n-1} a_kb_k, \quad 0\leq l < n-1.
\end{align*}
The real hyperquadric of signature \(l\) in \(\mathbb{C}^n\) is the real hypersurface
of real dimension \(2n-1\), denoted by  $\mathbb{H}^{2n-1}_l$, and is defined as follows:
\begin{equation}
  \mathbb{H}^{2n-1}_l
  =\left\{(z,w)=(z_1,\dots ,z_{n-1},w)\in\mathbb{C}^{n} \mid \Im  w-\langle z,\overline{z}\rangle_l=0\right\}.
\end{equation}
Of our special interest is the following 
real hypersurface in \(\mathbb{C}^{n+1}\), which has real dimension \(2n+1\) and is denoted
by $\mathcal{X}^{2n+1}_l$. Namely, we define
\begin{equation}
  \mathcal{X}^{2n+1}_l=\left\{(z,\zeta,w)\in\mathbb{C}^{n+1}
  \mid (1-|\zeta|^2)\Im  w 
  - \langle z,\overline{z}\rangle_l - \Re \left(\bar{\zeta} zz^t\right)=0, \ |\zeta|^2 < 1\right\},
\end{equation}
where, as above, \(z = (z_1, z_2, \dots, z_{n-1})\) and \(z^t\) is its transpose,
so that \(z z^t = z_1^2+ \cdots + z_{n-1}^2\).
In the case \(n = 2\) and \(l = 0\), this local model was found by Gaussier and Merker, which 
was then shown to be locally equivalent to the tube over the future light cone 
by Fels--Kaup \cite{fels2007cr}. Based on this work of Fels--Kaup, we can easily construct a local equivalence of \(\mathcal{X}_l^{2n+1}\) for \(n \geq 3\) and the tube over a null cone of a symmetric form in \(\mathbb{R}^{n+1}\).
For example, when \(n = 3\) and \(l = 1\), the model \(\mathcal{X}_1^{7}\) is locally equivalent 
to the tube \(\mathcal{T}_1\) over the null cone of a symmetric form 
in \(\mathbb{R}^4\) given by
\[
	\mathcal{C}_1 
	= \left\{(x_1, x_2, x_3, x_4) \in \mathbb{R}^4 \mid -x_1^2 + x_2^2+x_3^2 -x_4^2 = 0\right\}.
\]
More precisely, as a real hypersurface in \(\mathbb{C}^4\), 
\(\mathcal{T}_1 = i\mathbb{R} \times \mathcal{C}_1\) is given by
\begin{equation}\label{tflc1}
	\rho(z_1, z_2, z_3, w) := -\left(\Re w\right)^2 - \left(\Re z_1\right)^2 + \left(\Re z_2\right)^2 +\left(\Re z_3\right)^2,
\end{equation}
while the rational map
\begin{equation}\label{t2m}
	(z,w) \mapsto \left(\frac{2iz_1}{1+w-z_3},\ \frac{2z_2}{1+w-z_3}, \ \frac{1-w+z_3}{1+w-z_3}, \ \frac{2i(w+w^2 +z_1^2 - z_2^2 + z_3 - z_3^2)}{1+w-z_3}\right)
\end{equation}
provides a local equivalence of the germ 
\((\mathcal{T}_1, p = (0, 0, -1/2, 1/2))\) to \((\mathcal{X}, 0)\). 
Since both models are locally homogeneous, equivalence of a pair of arbitrary points
implies the equivalences of all pairs.
Note also that, for
the case \(l > 0\), the hypersurface \(\mathcal{X}_l^{2n+1}\) is no longer pseudoconvex. 
But it still plays an important
role in the study of real hypersurfaces in complex spaces, especially in that of homogeneous 2-nondegenerate homogeneous CR manifolds in 
several recent papers by Gregorovi\v{c}, Sykes, Porter--Zelenko, Santi, and others,
see, e.g., Gregorovi\v{c} \cite{g21}, Gregorovi\v{c}--Sykes \cite{gk23}, Porter--Zelenko \cite{p21}, and Santi \cite{santi20}, and many references therein. It is known
that \(\mathcal{X}^7_1\) has the \textit{third} largest dimension of 
symmetry algebras (dimension 15) among finitely nondegenerate real hypersurfaces in \(\mathbb{C}^4\). This differs a bit from the case of tube over the future light cone in \(\mathbb{C}^3\) whose symmetry algebras has \textit{second} largest dimension, cf. Fels--Kaup \cite{fels2007cr}.

The complex unit ball with signature \(l\), denoted by $\mathbb{B}^n_l$, has also been much studied in the literature.
In nonhomogeneous coordinates of \(\mathbb{C}^n\), it is defined by
\begin{align*}
\mathbb{B}^n_l
=\{z=(z_1, \dots ,z_{n-1}, w)\in\mathbb{C}^n \mid 1 - |w|^2 - \langle z,\overline{z}\rangle_l>0\}.
\end{align*}
When \(l = 0\), it is the unit ball in \(\mathbb{C}^n\), but when \(l>0\), it is
unbounded. The boundary of $\mathbb{B}^n_l$ is locally equivalent to the hyperquadric of signature $l$
via the Cayley transform:
\[
	(z, w) \mapsto \left(\frac{2z}{w + i}, \frac{w-i}{w+i}\right).
\]
Of our interest is also the ``signature version'' of 
the Lie ball, or generalized type IV domain, defined by
\begin{align*}
  D^{\mathrm{IV}}_{m,l}
  =\left\{z=(z_1, \dots , z_m)\in\mathbb{C}^m \mid 1-2\langle z,\overline{z}\rangle_l
  + \left|zz^t\right|^2 >0 \text{ and } \langle z,\overline{z}\rangle_l<1\right\}.
\end{align*}
The generalized Lie ball is also unbounded when \(l >0\) (for example, the unbounded null cone \(\{\langle z, \bar{z}\rangle_l = 0\}\) is contained in \(D^{\mathrm{IV}}_{m,l}\)). But it also possesses many 
interesting properties as the bounded Lie ball (i.e., \(l=0\)) does. Unfortunately, to 
the authors' best knowledge, the generalized Lie ball of signature \(l>0\) defined
as above has not been studied in the literature. Let us point out that 
the smooth boundary part of $D^{\mathrm{IV}}_{n+1,l}$ is locally CR equivalent to the local model \(\mathcal{X}^{2n+1}_l\).
 
We denote the CR automorphism 
group of a real hypersurface \(M\) or the biholomorphism group of a domain \(D\) by \(\Aut(M)\)
and \(\Aut(D)\), respectively. For example, for the hypersurfaces
$\mathbb{H}^{2n-1}_l\) and \(\mathcal{X}^{2n+1}_l\) as well as the complex domains 
\(\mathbb{B}^n_l\) and \(D^{IV}_{m,l}$,
their CR automorphism
groups and biholomorphism groups are
$\Aut(\mathbb{H}^{n}_l)\) and \(\Aut(\mathcal{X}^{n+1}_l)\) as well as \(\Aut(\mathbb{B}^n_l)\)
and \(\Aut(D^{IV}_{m,l})$, respectively.
Similarly, the stability group at a point \(p\in M\) of \(M\) consists of local CR automorphisms of \(M\) 
fixing $p$ is denoted
by \(\Aut_p(M)\).
If $A$ and $B$ are two real hypersurfaces, then two germs of CR maps $(H, p)$ and 
$(H', q)$ from $A$ to $B$ are considered to be equivalent if there exists 
$\psi\in \Aut(A), \varphi\in \Aut(B)$ such that $\psi(p) = q, \varphi(H(p))=H'(q)$
and 
\[
H'=\varphi\circ H\circ\psi^{-1}.
\]
Now, we state the main result of this paper.  
In the statement below, we write $z=(z_1,\dots,z_{n-1})$ and 
$f(z,w)=\left(f_1(z,w),\dots, f_{n-1}(z,w)\right)$.

\begin{theorem}
\label{thrmHn1Xn2}
Let $n\geq 3, 1\leq l<n-1$ and $U$ be an open subset of 
$\mathbb{H}^{2n-1}_l\subset\mathbb{C}^{n}$. Let $H\colon U\to\mathcal{X}^{2n+1}_l\subset\mathbb{C}^{n+1}$
be a transversal CR map. Then
\begin{enumerate}
\renewcommand{\labelenumi}{\Roman{enumi}.}
    \item 
    If $n=3$ and $l=1$, then $H$ is equivalent to exactly one of the germs at
    the origin of the following maps:
\begin{enumerate}
\renewcommand{\labelenumii}{(\roman{enumii})}
\item
$H_j(z,w)
= \left(\dfrac{z + i w z P_j}{1 + \varepsilon_j w^2},\
\dfrac{2 z P_j z^t}{1 + \varepsilon_j w^2 },\ \dfrac{w}{1 + \varepsilon_j w^2}\right)$,
where \(z = (z_1, z_2)\), \(z^t\) is its transpose, and \(P_j\) is one of the following five matrices
\[
P_1 = \begin{pmatrix} 0 & 0 \\ 0 & 0 \end{pmatrix},\ P_2 = \begin{pmatrix} 1 & i \\ i & -1 \end{pmatrix}, 
\ P_3 = \begin{pmatrix} -1 & -i \\ -i & 1 \end{pmatrix},
\]
\[
\ P_4 = \begin{pmatrix} 1 & 0 \\ 0 & -1 \end{pmatrix},
\ P_5 = \begin{pmatrix} 0 & i \\ i & 0 \end{pmatrix},
\]
and \(\varepsilon_j = \det(P_j) \in \{-1, 0, 1\}\) is the determinant of \(P_j\).

\item
$I(z,w)=\dfrac{2(z, w, w)}{1+\sqrt{1-4i\left(zz^t-iw^2\right)}}$.
\end{enumerate}

\item
If $n\geq 4$, then $H$ is equivalent to exactly one of the germs at the origin of the following maps:
\begin{enumerate}
\renewcommand{\labelenumii}{(\roman{enumii})}
\item 
$\ell^{(n)}(z,w)=(z, 0, w)$,

\item 
$I^{(n)}(z,w)=\dfrac{2(z,w,w)}{1+\sqrt{1-4i\left(z z^t-iw^2\right)}}$,
where \(z = (z_1, z_2, \dots, z_{n-1})\) and \(z^t\) is its transpose.
\end{enumerate}
\end{enumerate}
\end{theorem}
\begin{remark}
\rm
The maps \(\ell^{(n)}\) and \(I^{(n)}\), which already appeared in the signature
zero case, are isometries of the  ``canonical'' indefinite K\"ahler metric of the
one-sided neighborhoods of the source and target. They are inequivalent since
all automorphisms are rational. The inequivalences of pairs of maps related to polynomial
maps \(H_{2}\) and \(H_{3}\)
and rational maps \(H_{4}\) and \(H_{5}\) can be seen from 
the ``geometric rank'' and the CR Ahlfors tensor: Both \(\ell^{(n)}\)
and \(I^{(n)}\) have vanishing geometric rank, \(H_{2}\) and \(H_{3}\) have geometric rank one
while \(H_{4}\) and \(H_{5}\) have geometric rank two. The CR Ahlfors tensor of
\(H_{2}\) is nonnegative while that of \(H_{3}\) is nonpositive. Likewise, the CR
Ahlfors tensor of \(H_{4}\) is positive definite
while that of \(H_{5}\) has two eigenvalues of opposite sign.
This explains the pairwise inequivalence of all six maps in the case \(n=3\) and of two maps in 
the case \(n \geq 4\). We will discuss this in more details later.
\end{remark}

By a similar strategy as in Reiter--Son \cite{ReiSon2022, ReiSon2024},
we obtain from Theorem~\ref{thrmHn1Xn2} a list of local proper holomorphic maps from the generalized ball
into the generalized Lie ball.

\begin{corollary}\label{thrmBnlD44l}
Let $n\geq 3, 1\leq l< n-1$ and $F\colon \mathbb{B}^n_l\to D^{IV}_{n+1,l}$ be a proper
holomorphic map which extends smoothly to some boundary point
$p\in\partial\mathbb{B}^n_l$.
\begin{enumerate}
\renewcommand{\labelenumi}{\Roman{enumi}.}
\item
If $n=3$ and $l=1$, then $F$ is equivalent to exactly one of the following maps:

\begin{enumerate}
\renewcommand{\labelenumii}{(\roman{enumii})}
\item
$R_0(z,w)
=\left(\dfrac{z_1}{\sqrt{2}},\dfrac{z_2}{\sqrt{2}},\dfrac{2w+2w^2-z_1^2-z_2^2}{4(1+w)},\dfrac{i\left(2w+2w^2+z_1^2+z_2^2\right)}{4(1+w)}\right)$,

\item 
$R_1(z,w)=\left(\dfrac{zA+wzB}{\sqrt{2}(1+w)},\, \dfrac{w}{2}+\dfrac{zFz^t}{4(1+w)},\, i\left(\dfrac{w}{2}-\dfrac{zFz^t}{4(1+w)}\right)\right)$,
where
\[
A=\begin{pmatrix}
            1 & 2i \\
            -2i & 3
        \end{pmatrix}, 
B= \overline{A^{-1}}
=\begin{pmatrix}
            -3 & -2i \\
            2i & -1
\end{pmatrix},\
\text{and} \
F=\begin{pmatrix}
  -5 & 4i \\
  4i & 3
\end{pmatrix}.
\]
\item
$R_2(z,w)=\left(\dfrac{-zB-wzA}{\sqrt{2}(1+w)},\, \dfrac{w}{2}+\dfrac{zF^{-1}z^t}{4(1+w)},\, i\left(\dfrac{w}{2}-\dfrac{zF^{-1}z^t}{4(1+w)}\right)\right)$, where \(A, B\), and \(F\)
are as above.

\item
$P_1(z,w)=\left(wz_1,z_2,\dfrac{w^2-z_1^2}{2},\dfrac{i(w^2+z_1^2)}{2}\right)$,

\item 
$P_2(z, w) = \left(z_1, w z_2, \dfrac{w^2 - z_2^2}{2}, \dfrac{i(w^2 + z_2^2)}{2}\right)$.

\item
$I(z,w)
=\left(\dfrac{z_1}{\sqrt{2}},\dfrac{z_2}{\sqrt{2}},\dfrac{w}{\sqrt{2}},\dfrac{1-\sqrt{1-z_1^2-z_2^2-w^2}}{\sqrt{2}}\right)$.
\end{enumerate}
\item
If $n\geq 4$, then $F$ is equivalent to exactly one of the following maps:
\begin{enumerate}
\renewcommand{\labelenumii}{(\roman{enumii})}
\item
$R_0^{(n)}(z,w)=\left(\dfrac{z}{\sqrt{2}},\dfrac{2w+2w^2-z z^t}{4(1+w)},\dfrac{i\left(2w+2w^2+z z^t\right)}{4(1+w)}\right)$,

\item
$I^{(n)}(z,w)
=\left(\dfrac{z}{\sqrt{2}},\dfrac{w}{\sqrt{2}},\dfrac{1-\sqrt{1-z z^t-w^2}}{\sqrt{2}}\right)$.
\end{enumerate}
\end{enumerate}
\end{corollary}
\begin{remark}
\rm
The maps \( R_0^{(n)} \) and \( I^{(n)} \) were previously studied for all 
\( n \geq 2 \) in the work of Xiao--Yuan~\cite{XiaoYuan2020}. These maps are isometries (up to a conformal constant factor) with respect to the Bergman metrics on the unit ball and the Lie ball. Both maps exhibit singularities at some points on the boundary of the unit ball. 

When \( l > 0 \), the complex variety
\[
W = \{(z, w) \in \mathbb{C}^n \mid w + 1 = 0\},
\]
on which \( R_0^{(n)} \) is singular, intersects both sides of the boundary
\( \partial \mathbb{B}_l^n \). Consequently, neither \( R_0^{(n)} \) nor
\( I^{(n)} \) maps the entire domain \( \mathbb{B}_l^n \) into
\( D^{\mathrm{IV}}_{n+1,l} \). Yet, these two maps are still
local isometries of ``canonical'' pseudo-K\"ahler-Einstein metrics of the
generalized ball and Lie ball.
Similar behavior is observed for the rational maps \( R_1 \) and \( R_2 \). Thus,
in the case \(l>0\) we have examples of local isometries of the pseudo-K\"ahler metrics
that do not extend to a global one.

The
polynomial maps \(P_1\) and \(P_2\)
have similar formulas, but they are indeed inequivalent, as can be observed by comparing their CR Ahlfors tensors. Although  
\(P_1\) is polynomial, it does not send the whole \(\mathbb{B}^3_1\)
into \(D^{\mathrm{IV}}_{4,1}\). The map \(P_2\) sends
\(\mathbb{B}^3_1\) into the open set
\[
	\widetilde{D}^{\mathrm{IV}}_{4,1} := \left\{z=(z_1, \dots , z_4)\in\mathbb{C}^4 \mid 1-2\langle z,\overline{z}\rangle_1
	  + \left|zz^t\right|^2 >0 \right\},
\]
and sends \(\mathbb{B}^3_1 \cap \{|z_1|^2< 1\}\) into \(D^{\mathrm{IV}}_{4,1}\).
On the other hand, when being restricted to the ``slice'' \(\mathbb{B}^3_1 \cap \{z_1 = 0\} \simeq \mathbb{B}^2\), two maps \(P_{1}\) and \(P_{2}\) 
induce two polynomial maps from \(\mathbb{B}^2\) into \(D^{\mathrm{IV}}_{3}\) discovered by Reiter--Son \cite{ReiSon2022}.

The rational maps \(R_1\) and \(R_2\) did not appear earlier in the literature.
They have ``geometric rank'', as defined in Section~\ref{sect:grk}, equal one.
We should point out that rank one maps do not exist in the case \(l=0\) considered in Reiter--Son \cite{rs24}.
\end{remark}
The rest of the paper is organized as follows. 
In section~\ref{sec:nf}, we will provide the formulas for some automorphisms 
in $\mathbb{H}^{2n-1}_l$ and $\mathcal{X}^{2n+1}_l$, and use them 
to normalize the CR maps from $\mathbb{H}^{2n-1}_l$ to $\mathcal{X}^{2n+1}_l$. 
In section~\ref{sect:3}, we will prove Theorem \ref{thrmHn1Xn2}. In section \ref{sect:4}, we will present 
a proof of Corollary \ref{thrmBnlD44l} as well as some further extensions of
Theorem \ref{thrmHn1Xn2}. Finally, in section \ref{sect:5}, we will
discuss the higher (co-)dimensional case. 
\section{Normalization for Theorem \ref{thrmHn1Xn2}}
\label{sec:nf}
\subsection{The CR automorphism groups of $\mathbb{H}^{2n-1}_l$ and $\mathcal{X}^{2n+1}_l$}
The automorphism group of a Levi-nondegenerate hyperquadric $\mathbb{H}^{2n-1}_l$ 
is well-known and parametrized as follows:

\begin{align*}
 \phi_{c,r,s,U}(z,w)= \frac{\left(s(z + cw)U,s^2w\right)}{1+rw-iw\langle c,\overline{c}\rangle_l-2i\langle\overline{c},z\rangle_l},
\end{align*}
where $c\in\mathbb{C}^{n-1}, r\in\mathbb{R}, s>0$ and $U\in SU(l,n-1)$, see Chern--Moser \cite{cm74}. 

Next, we consider the automorphisms in $\mathcal{X}^{2n+1}_l$. They can be computed
explicitly by integrating infinitesimal CR automorphisms, which are given by: %
\begin{align*}
X_{-2} & = \frac{\partial}{\partial w},\\
X_{-1} & = a \cdot \frac{\partial}{\partial z} -\zeta \left\langle \bar a, \frac{\partial}{\partial z} \right\rangle_l + 2 i \left\langle \bar a, z \right\rangle_l \frac{\partial}{\partial w}, \qquad a \in \mathbb C^{n-1},\\
X_0^1 & = z \cdot \frac{\partial}{\partial z} +2 w \frac{\partial}{\partial w},\\
X_0^2 & = i z \cdot \frac{\partial}{\partial z} +2 i \zeta \frac{\partial}{\partial \zeta},\\
X_0^3 & = zA \cdot \frac{\partial}{\partial z}, \qquad A=(a_{ij})_{1\leq i,j\leq n-1}, \quad a_{ii}=0, \\
& \qquad  i\neq j: a_{ij} + a_{ji} = 0, \epsilon_i a_{ij} + \epsilon_j \bar a_{ji} = 0,\\
X_0^4 & = -\bar b \zeta \left(z \cdot \frac{\partial}{\partial z}\right) + (b-\bar b \zeta^2) \frac{\partial}{\partial \zeta} + i \bar b zz^t \frac{\partial}{\partial w}, \qquad b \in \mathbb C, \\
X_1 & = (i w \zeta - z z^t) \left(c \cdot \frac{\partial}{\partial z}\right) + 2 (c \cdot z) \left(z \cdot \frac{\partial}{\partial z} \right) + i w \left\langle \bar c, \frac{\partial}{\partial z} \right \rangle_l \\
& \qquad  + 2 \bigl[ \left\langle \bar c, z \right\rangle_l + \zeta (c \cdot z) \bigr] \frac{\partial}{\partial \zeta} +2 (c \cdot z) w \frac{\partial}{\partial w}, \qquad c \in \mathbb C^{n-1},\\
X_2 & = w \left( z \cdot \frac{\partial}{\partial z} \right) 
- i z z^t \frac{\partial}{\partial \zeta} + w^2 \frac{\partial}{\partial w}.
\end{align*} 
The Lie algebra generated by the vector fields above which vanish at the origin has real dimension $(n^2+2n+4)/2$. The stability group \(\Aut_0(\mathcal{X}^{2n+1}_{l})\)  can be computed by integrating the vector
fields in its symmetry algebra which vanish at the origin. 
To make our formulas more concise, we denote by
\[
D_l = \mathrm{diag}(\underbrace{-1, \dots, -1}_{l}, \underbrace{1, \dots , 1}_{n-1-l}) \in \mathrm{Mat}(n-1; \mathbb{C})
\]
the diagonal matrix with signature \(l\). Moreover, put
\[ 
	\delta 
	= \delta(z,w) 
	= 1 - (r'+i a D_l \bar{a}^t) w - 2i z D_l \bar{a}^t+ i\, \overline{aa^t} (w\zeta + i z z^t),
 \]
 for \(a = (a_1, a_2) \in \mathbb{C}^2\), \(u'\in \mathbb{C}\), \(|u'| = 1\), \(r' \in \mathbb{R}\), and \(P \in \mathrm{O}(l, n-l-1)\). The stability group consists of holomorphic maps of the form 
 \(\gamma = (\eta, \gamma_n, \gamma_{n+1})\), where \(\eta = (\gamma_1,\gamma_2, \dots, \gamma_{n-1})\) and 
\begin{align*}
\eta(z,w) & = s' u' (z+ w a - (w\zeta + i z z^t)D_l\bar{a})P/\delta,  \\
\gamma_n(z,w) & = {u'}^2\left(\zeta -2 z a^t -i aa^t w  -(r'-i \bar{a}D_l a^t )(w\zeta + i z z^t)\right)/\delta, \\
\gamma_{n+1}(z,w) & = {s'}^2 w/\delta,
\end{align*}
with $s'>0$.
This form of the CR automorphisms is similar to the case of the tube over the future light cone in \(\mathbb{C}^{n+1}\).
We parametrize elements of \(\Aut_0(\mathcal{X}^{2n+1}_{l})\) by
\begin{align*}
(z,w) \mapsto \psi'_{s',u',P,a,r'}(z,w) = \gamma(z,w).
\end{align*}
This parametrization is used in the normalization discussed in the next 
section.

\label{ss:2.1}
\subsection{A partial normalization}
Let \(U \subset \mathbb{C}^n\) be a connected open neighborhood of the origin and $H=(f, \phi, g)=(f_1, f_2, \dots, f_{n-1}, \phi, g)$ a holomorphic map from
$U$ to \(\mathbb{C}^{n+1}\), \(H(0) = 0\), and $H(U\cap\mathbb{H}^{2n-1}_l)\subset\mathcal{X}^{2n+1}_l$.
Then the following ``mapping equation'' holds:
\begin{align}
\label{e:mapeq}
& \Bigl(g(z,\overline{w}+2i\langle z,\overline{z}\rangle_l)
  -\overline{g}(\overline{z},\overline{w})\Bigr)
\Bigl(1-\phi(z,\overline{w}+2i\langle z,\overline{z}\rangle_l)\overline{\phi}(\overline{z},\overline{w}) \Bigl) \notag\\
& \quad -2i\Bigl\langle f(z,\overline{w}+2i\langle z,\overline{z}\rangle_l),\overline{f}(\overline{z},\overline{w}) \Bigr\rangle_l  -i\overline{\phi}(\overline{z},\overline{w})F(z,\overline{w}+2i\langle z,\overline{z}\rangle_l) \notag  \\
& \quad -i\phi(z,\overline{w}+2i\langle z,\overline{z}\rangle_l)\overline{F}(\overline{z},\overline{w})
=0,
\end{align}
for all \((z, \bar{z}, \bar{w})\) in a suitable neighborhood of the origin in 
\(\mathbb{C}^{n-1}\times \mathbb{C}^{n-1} \times \mathbb{C}\). Here,
we put \(F = f_1^2 + \cdots + f_{n-1}^2\) for short.

If $H \colon U\cap \mathbb{H}^{2n-1}_l \to \mathcal{X}^{2n+1}_l$, \(H(0) = 0\)
is a smooth CR map and if \(l > 0\), then by the Lewy extension theorem (see \cite{ber99}), \(H\) is the restriction
of a holomorphic map in a neighborhood of the origin, which we also denote by \(H\). 
Then \eqref{e:mapeq} holds for this extended map.%

Using the stability groups at the origin of the source and the target, we can bring the map of interest into the following partial normal form.

\begin{claim}
\label{Claim normal1}
Let $(H, p)$ be a germ at $p \in\mathbb{H}^{2n-1}_l$ of a smooth
transversal CR map which sends the germ at $p$ of $\mathbb{H}^{2n-1}_l$ into 
$\mathbb{\mathcal{X}}^{2n+1}_l$. Then $(H, p)$ is equivalent to the germ at the 
origin of a CR map $\tilde{H} = (f, \phi, g)$ which is of the following form:
\begin{equation}
\label{e:normalform}
  \begin{cases}
  f(z,w) & = z+\dfrac{i}{2}w(z{A} )+vw^2+O(3), \\
  \phi(z,w) & = \lambda w+z{B} z^t+wz\mu^t+\sigma w^2 +O(3), \\
  g(z,w) & = w+O(3),
  \end{cases}
\end{equation}
where $A, B \in \mathrm{Mat}(n-1;\mathbb{C}), 
v=(v_1, \dots ,v_{n-1})\in\mathbb{C}^{n-1}$, $\mu=(\mu_1,...,\mu_{n-1})\in\mathbb{C}^{n-1}$,
and $\lambda,\sigma\in\mathbb{C}$. 
Moreover, $B = [b_{jk}] $ is symmetric with \(b_{11}\) being real.
\end{claim}
\begin{proof} 
Applying $\overline{z}=\overline{w}=0$ in the mapping equation \eqref{e:mapeq} yields $g(z,0)=0$.
Taking the derivative of the mapping equation with respect to $\overline{w}$ and
applying $\overline{z}=\overline{w}=0$ implies that $\dfrac{\partial g}{\partial w}(0,0)\in \mathbb{R}$.
The transversality of $H$ implies that $\dfrac{\partial g}{\partial w}(0, 0)\ne 0$.

Let 
\[
E_k=\left(\dfrac{\partial f_j}{\partial z_k}\right)_{1\leq j\leq n-1}.
\]
Taking the second-order derivatives of the mapping equation \eqref{e:mapeq}
with respect to the 
variables
$z_k$ and $\overline{z_j}$ with $1\leq k, j\leq n-1$ and setting $z=\overline{z}=\overline{w}=0$,
we obtain that:
\begin{equation}
\left\langle E_k,\overline{E_j}\right\rangle_l
=\frac{\partial g}{\partial w}(0,0)\epsilon_k\delta^k_j,
\end{equation}
where \(\epsilon_k : = \sign(k-l-1/2)\).

Let $E \in \mathrm{Mat}(n-1; \mathbb{C})$ be the matrix with $E_k$ being its $k^{\mathrm{th}}$ column. Then we can 
deduce that
\begin{equation}
  \overline{E}^TD_{l}E=\dfrac{\partial g}{\partial w}(0,0)I_{l}.
\end{equation}

Now we write 
$H_{[n]}=\varphi'_n\circ H_{[n-1]}\circ\varphi_n$, where 
$H_{[0]}=H$ and $\varphi_{n}, \varphi'_{n}$ are two suitable automorphisms of 
$\mathbb{H}^{2n-1}_l$ and $\mathbb{\mathcal{X}}^{2n+1}_l$, respectively. Taking 
the derivatives with respect to the variables $z_i$ for
$1\leq i\leq n-1$ and setting $z=w=0$, we obtain that:
\[
\frac{\partial H_{[1]}}{\partial z}(0,0)
=\left(ss'u'EU,su'^2\left(\dfrac{\partial \phi}{\partial z}(0,0)U-2iaEU\right),0\right).
\]
After exchanging components, if needed, we can assume that 
$\dfrac{\partial g}{\partial w}(0,0)>0$. Now choose 
\[
s'=1/{\left(s\sqrt[n-1]{\dfrac{\partial g}{\partial w}(0,0)}\right)}, \quad
U=\overline{u'}\,E^{-1},
\]
we can assume that $\dfrac{\partial f_i}{\partial z_j}(0,0)=\delta^i_j$.
Considering $H_{[2]}=\varphi'_2\circ H_{[1]}\circ\varphi_2$ with the above parameters,
we obtain that:
\[
\dfrac{\partial H_{[2]}}{\partial z}(0,0)=\left(I_{n-1},su'\left(-2ia+\dfrac{\partial\phi}{\partial z}(0,0)\right),0\right).
\]
Choosing $a=-\dfrac{i}{2}\dfrac{\partial\phi}{\partial z}(0, 0)$ and considering
$H_{[3]}=\varphi'_3\circ H_{[2]}\circ\varphi_3$ with the above parameters, we obtain that:
\[
\dfrac{\partial H_{[3]}}{\partial z}(0,0)=\left(I_{n-1},0,0\right).
\]
Now taking the derivative with respect to the variable $w$ of $H_{[3]}$ at $z=w=0$,
we obtain the following:
\[
\dfrac{\partial H_{[3]}}{\partial w}(0,0)=\left(c+su'\dfrac{\partial f}{\partial w}(0,0),s^2u'^2\dfrac{\partial \phi}{\partial w}(0,0),1\right).
\]
Choosing $c=-su'\dfrac{\partial f}{\partial w}(0,0)$, we obtain that:
\[
\dfrac{\partial H_{[3]}}{\partial w}(0,0)
=\left(0,s^2u'^2\dfrac{\partial \phi}{\partial w}(0,0),1\right).
\]
Now we will consider the second-order derivatives of $H_{[3]}$ at $(0,0)$.
Taking the derivatives with respect to $z_k$ twice and $\overline{z_k}$ once of
the mapping equation at $z=\overline{z}=\overline{w}=0$, we have:
\begin{equation}\label{derif1}
    \dfrac{\partial^2f_k}{\partial z_k^2}(0, 0)
    =2 \dfrac{\partial^2 g}{\partial z_i\partial w}(0,0) \quad \forall k = 1, \dots, n-1.
\end{equation}
Taking the derivatives with respect to $z_k$ and $\overline{w}$ of the mapping equation
at $z=\overline{z}=\overline{w}=0$, we have:
\begin{equation}\label{derif2}
    \dfrac{\partial^2 g}{\partial z_k\partial w}(0,0)=0 \quad \forall k = 1,\dots, n-1.
\end{equation}
Taking the derivatives with respect to $z_k,z_j$ and $\overline{z_k}$ with $k\ne j$
of the mapping equation at $z=\overline{z}=\overline{w}=0$, we have:
\begin{equation}\label{derif3}
    \dfrac{\partial^2f_k}{\partial z_k\partial z_j}(0,0)
    =\dfrac{\partial^2 g}{\partial z_k\partial w}(0,0) \quad \forall i,j=1,\dots,n-1,\ i\ne j.
\end{equation}
Taking the derivatives with respect to $z_k$, $z_j$, and $\overline{z_q}$ with 
$k\ne q,j\ne k$ of the mapping equation at $z=\overline{z}=\overline{w}=0$, we have:
\begin{equation}\label{derif4}
    \dfrac{\partial^2f_q}{\partial z_k\partial z_j}(0,0)
    =0 \quad \forall k,j,q = 1,\dots, n-1,\ k\ne q, j\ne q.
\end{equation}
Combining (\ref{derif1}), (\ref{derif2}), (\ref{derif3}), (\ref{derif4}), we obtain
that all second-order derivatives of $f(z,w)$ with respect to $z_k$ and $z_j$ vanish at $0$.

Now considering $H_{[4]}=\varphi'_4\circ H_{[3]}\circ\varphi_4$ and 
taking the second-order derivatives with respect to
$w$ at $z=w=0$, we obtain
\[
\dfrac{\partial^2 g_{[4]}}{\partial w^2}(0,0) = -2r+s^2\left(2t+\dfrac{\partial^2 g}{\partial w^2}(0,0)\right),
\]
such that when choosing $r=s^2\left(2t+\dfrac{\partial^2 g}{\partial w^2}(0,0)\right)/2$, we can assume that $\dfrac{\partial^2 g}{\partial w^2}(0,0)=0$.
Furthermore, taking the second-order derivatives of $\phi_{[4]}$ with respect to $z_1$,
we obtain that: 
\[
\dfrac{\partial^2 \phi_{[4]}}{\partial z_1^2}(0,0) = s^2\left(-2it+\dfrac{\partial^2\phi}{\partial z_1^2}(0,0)\right).
\]
After choosing $t=\Im \dfrac{\partial^2\phi}{\partial z_1^2}(0,0)/2$ we can assume that
$\dfrac{\partial^2\phi}{\partial z_1^2}(0,0)\in\mathbb{R}$, showing that \(b_{11}\) is real. The proof is complete.
\end{proof}

\begin{claim}\label{Claim normal2}
    The partial normal form of a map germ \((H,p)\) obtained in Proposition \ref{Claim normal1} satisfies $A=B$.
\end{claim}
\begin{proof}
We denote the entries of $A$ and $B$ by $a_{ij}$ and $b_{ij}$, respectively. 
Taking the derivative of the mapping equation with respect to the variables 
$z_k,z_k,\overline{z_k},\overline{z_j}$ with $1\leq k\ne j\leq n-1$ at $z=\overline{z}=\overline{w}=0$, 
and using the fact that ${B} $ is a symmetric matrix, we have:
\begin{equation}\label{entries1}
    a_{kj}=\epsilon_{k}\epsilon_{j}\overline{b_{kj}}
    =\epsilon_{k}\epsilon_{j}\overline{b_{jk}}=a_{jk}.
\end{equation}
This shows that ${A} $ is symmetric.

Taking the derivatives of the mapping equation with respect to the variables 
$z_k, z_j, \overline{z_k}$, and $\overline{z_k}$ at $z=\overline{z}=\overline{w}=0$, we have:
\begin{equation}\label{entries2}
    a_{kj}=b_{kj}=b_{jk}=a_{jk}.
\end{equation}
From (\ref{entries1}) and (\ref{entries2}) we have $a_{kj}=b_{kj}$ is real if
$\epsilon_{k}\epsilon_{j}=1$, and $a_{kj}=b_{kj}$ is purely imaginary if $\epsilon_{k}\epsilon_{j}=-1$.
Here, as above, \(\epsilon_j = \sign(j - l - 1/2)\).

Taking the fourth-orders derivative of the mapping equation with respect to the
variables $z_k,z_k,\overline{z_k}$, and $\overline{z_k}$ at $z=\overline{z}=\overline{w}=0$, 
we have:
\begin{equation}\label{entries3}
    2a_{kk}=b_{kk}+\overline{b_{kk}}.
\end{equation}
Thus $a_{kk}\in\mathbb{R}$ for all $1 \leq k \leq n-1$. 
Taking the fourth-order derivatives of the mapping equation with respect to the variables
$z_k,z_k,\overline{z_j},\overline{z_j}$ with $1\leq k\ne j\leq n-1$ at 
$z=\overline{z}=\overline{w}=0$, we have:
\begin{equation}\label{entries4}
  \overline{b_{kk}}+b_{jj}=0
\end{equation}
From (\ref{entries3}) and (\ref{entries4}) we deduce that:
\begin{equation}\label{entries5}
    a_{kk}+a_{jj}=0, \quad\forall\  1\leq k\ne j\leq n-1.
\end{equation}
From here we consider two cases depending on the value of $n$:

\textbf{Case 1:} If $n\geq 4$: Solving the system of equations (\ref{entries5})
we obtain $a_{kk}=0$ for all $1 \leq k \leq n-1$. 
Applying (\ref{entries3}), (\ref{entries4}) and the fact that 
$b_{11}=\dfrac{\partial^2\phi}{\partial z_1^2}(0,0)\in\mathbb{R}$ we have 
$b_{kk}=0$ for all $1 \leq k \leq n-1$. Thus we have
${A} ={B} $ and the diagonal entries equal to $0$.

\textbf{Case 2:} If $n=3$: Solving the system of equations (\ref{entries5}) 
we obtain $a_{11}=-a_{22}$. Applying (\ref{entries3}), (\ref{entries4}) and the
fact that $b_{11}\in\mathbb{R}$ we have 
$b_{kk}=a_{kk}$ for $j =1, 2$. Thus we have 
$A = B$. The proof is complete.
\end{proof}

For $1\leq j,k\leq n-1$, we denote $a_{jk}=c_{jk}$ if $\epsilon_{j}\epsilon_{k}=1$,
and $a_{jk}=ic_{jk}$ if $\epsilon_{j}\epsilon_{k}=-1$.
From Proposition \ref{Claim normal2} we have 
$c_{jk}\in\mathbb{R}$ for all $1\leq j,k\leq n-1$. 
We will use the above normalization in Section~\ref{sect:3}. 
\subsection{Geometric rank and the CR Ahlfors tensor}
\label{sect:grk}
In this section, we analyze a tensor defined for each transversal CR map between
real hypersurfaces. In our situation, the tensor has an invariant property and 
provides an efficient way to prove the inequivalence of CR maps. Specifically, 
we shall use it in the proof of Corollary~\ref{thrmBnlD44l}.

The CR Ahlfors tensor was first defined for CR maps between strictly pseudoconvex
pseudohermitian manifolds by Lamel--Son \cite{lamel2019cr}. The original construction
is of differential geometric nature which is based on a construction of the 
conformal counterpart of Stowe \cite{stowe}. It was subsequently used in several papers, 
see Reiter--Son \cite{ReiSon2022,rs24,ReiSon2024}. 
\subsubsection{The CR Ahlfors tensor}

Let $G \colon M\to N$ be a CR transversal map extending holomorphically to a
neighborhood of $M$. By the CR transversality, there exists a real-valued smooth
function $v$ such that
\begin{equation}
\label{e:map1}
\rho_{N}\circ G=\pm e^v\rho_{M}.
\end{equation}
In many situations, the complex Hessian of \(v\) is well-behaved
and thus it is important for our purposes.

Let \(N = \partial D^{\mathrm{IV}}_{m,l}\) be the boundary of the generalized Lie ball and let
$\rho_{D^{\mathrm{IV}}_{m,l}}
=1-2\langle z,\overline{z}\rangle_l+\left|z z^t\right|$ be its defining
function. Also, let $M = \partial\mathbb{B}^n_l$ be the boundary of the
generalized ball with the defining function 
$\rho_{\mathbb{B}^n_l}=1-\langle z,\overline{z}\rangle_l$. Without loss of generality, 
we may assume that in \eqref{e:map1} the positive sign occurs. At each point \(p\), the complex Hessian of \(v\) 
gives rise to a Hermitian form on the holomorphic tangent space \(T^{(1,0)}_p \mathbb{C}^n\).
If \(p\in M$, we identify \(T^{(1,0)}_pM = T^{(1,0)}_p \mathbb{C}^n \cap \mathbb{C}T_p\mathbb{C}^n\).
By this, we can restrict the complex Hessian \(v_{Z\bar{Z}}\) (where \(Z = (z,w)\))
to obtain a tensor $\mathcal{A}(G)$ defined on $T^{(1, 0)}M$ and its complex conjugate.
Thus, 
\begin{align*}
    \mathcal{A}(G)
    =v_{Z\bar{Z}}\Big|_{T^{(1,0)}\partial \mathbb{B}^n_l\times T^{(0,1)} \partial \mathbb{B}^n_l}.
\end{align*}
Observe that 
$\mathcal{A}(G)$ is invariant under the pre-composition and composition with 
automorphisms of $\mathbb{B}^n_l$ and $D^{\mathrm{IV}}_{m,l}$, respectively. In fact,
this invariant property is a consequence of the fact that the CR automorphisms
of the source and the target preserves certain (indefinite) K\"ahler metrics defined
by K\"ahler potentials related to the defining functions.

Precisely, let $\psi$ be an 
automorphism of $D^{\mathrm{IV}}_{m,l}$ and let $G'=\psi\circ G$.
From the fact that 
$\psi\in \Aut\left(D^{\mathrm{IV}}_{m,l}\right)$ we can check directly that
\begin{align*}
\rho_{D^{\mathrm{IV}}_{m,l}}\circ\psi=|q|^2\rho_{D^{\mathrm{IV}}_{m,l}}
\end{align*}
for some rational function $q$ (which can be explicitly computed) having no pole and zero on $N$. Indeed, we can
verify this claim for \(\psi\) being one of the five maps \(\psi_j\) given 
in Section~\ref{ss:2.1} (which generate the full group of automorphisms via compositions). 
Thus, we have
\[
\rho_{D^{\mathrm{IV}}_{m,l}}\circ G'
=\rho_{D^{\mathrm{IV}}_{m,l}}\circ\psi\circ G
=\left(|q|^2\rho_{D^{\mathrm{IV}}_{m,l}}\right)\circ G
=|q\circ G|^2
\left(\rho_{D^{\mathrm{IV}}_{m, l}}\circ G\right)
=|q\circ G|^2e^v\rho_{B^n_l}.
\]
Hence, we obtain
\[
v'=v+\log\left(|q\circ G|^2\right).
\]
Since $q$ is holomorphic along \(N\), \(q\circ G\) is holomorphic along \(M\),
and thus the complex Hessian matrices of $v$ and $v'$ are the same, as desired.

Similarly, we can prove that the CR Alhfors tensor is invariant under the 
pre-composition with an automorphisms of $\mathbb{B}^n_l$, we leave the details
to the reader.

\subsubsection{Geometric rank}
As first noticed by Lamel and Son, the CR Ahlfors tensor defined in \cite{lamel2019cr}
is closely related to the notion
of geometric rank of Huang \cite{hu99}. This has been exploited in 
Reiter--Son \cite{ReiSon2022, ReiSon2024, rs24}. In what follows, we shall use this
to define the geometric rank for CR maps in our setting: For \(N = \mathcal{X}_l^{2n+1}\), we take
\[
  \rho'(z, \zeta, w) 
  = (1 - |\zeta|^2)\Im(w) - \langle z, \bar{z}\rangle_l - \Re(\bar{\zeta} z z^t),
\]
and for \(M = \mathbb{H}^{2n-1}_l\), we take
\[
  \rho = \Im (w) - \langle z, \bar{z}\rangle_l.
\]
Assume that \(H = (f, \phi, g)\) sends the germ at the origin of the hyperquadric into \(\mathcal{X}_l^{2n+1}\) and has a partial normal form \eqref{e:normalform}. Then there exists a real-valued function \(Q(z, w, \bar{z}, \bar{w})\) defined in 
a neighborhood of the origin in \(\mathbb{C}^n\) such that
\begin{align}
\label{e:mapeq2}
 (1 - & |\phi(z, w)|^2)  (g(z, w) - \bar{g}(\bar{z}, \bar{w})) 
 - 2i \langle f(z, w), \bar{f}(\bar{z}, \bar{w}\rangle _l \notag \\
 & - i(\phi(z, w) \bar{F}(\bar{z}, \bar{w}) + \bar{\phi}(\bar{z}, \bar{w}) F(z, w))
 = Q(z, w, \bar{z}, \bar{w})(w - \bar{w} - 2i \langle z, \bar{z}\rangle_l).
\end{align}
Comparing the coefficients of \(w\) of the Taylor expansions at the origin of both sides of \eqref{e:mapeq2}, we find that 
\(Q(0) = 1\). Setting \(w = \bar{w} = 0\), the equation reduces to
\begin{align*}
 - 2i \langle f(z, w), \bar{f}(\bar{z}, \bar{w}\rangle _l
 - i(\phi(z, 0) \bar{F}(\bar{z}, 0) + \bar{\phi}(\bar{z}, 0) F(z, 0))
 = -2i Q(z, 0, \bar{z}, 0)\langle z, \bar{z}\rangle_l.
\end{align*}
We expand \(Q(z, 0, \bar{z}, 0)\) in homogeneous terms of bidegree \((k, l)\):
\[
Q(z, 0, \bar{z}, 0) = 1 + q_{(1,0)}(z) + q_{(0, 1)}(\bar{z}) 
+ q_{(2,0)}(z, \bar{z}) + q_{(0,2)}(z, \bar{z}) + q_{(1,1)}(z, \bar{z}) + O(3).
\]
Comparing the terms of bidegrees (1, 2) and (2, 1), we find that
\[
  q_{(1,0)}(z) = q_{(0, 1)}(\bar{z}) = 0.
\]
Comparing the terms of bidegree (2, 2), we find that
\begin{equation}
\label{e:gauss}
q_{(1,1)}(z,\bar{z}) \langle z,\bar{z}\rangle_l
=
\frac{1}{2} \left((z B z^t)\bar{z} \bar{z}^t 
+ \overline{(z B z^t)} (z z^t) \right).
\end{equation}
Thus, if
\[
  q_{(1,1)}(z,\bar{z}) = \sum_{j, k = 1}^{n-1} q_{j\bar{k}} z_j \bar{z}_k,
\]
we easily find that
\[
\epsilon_j q_{j\bar{k}} = b_{jk},
\]
where, as before, \(\epsilon_j = \sign(j-l-1/2)\) (the signum function).
In particular, since \(b_{ij} \) is symmetric, we find that \(q_{j\bar{k}}\) is 
real if \(\epsilon_j = \epsilon_k\) and purely imaginary otherwise.

Since \(Q(0) = 1\) and all first order derivatives of \(Q\) at the origin vanish,
we can easily find that
\begin{equation}
\label{e:crahlfors}
  \mathcal{A}(H)\big|_0\left(\partial_{j}\big|_0, \ \partial_{\bar{k}}\big|_0\right) 
  = \left(\log(Q)\right)_{j\bar{k}}
  = q_{j\bar{k}}.
\end{equation}
We call the rank of \(\mathcal{A}(H)\bigl|_0\), or equivalently, the rank of
\([q_{j\bar{k}}]\), the \emph{geometric rank} of \(H\) at the origin.
\begin{remark}
\rm
In the case \(n \geq 4\), it follows from \eqref{e:gauss} and Huang's Lemma \cite{hu99}
that \(q_{(1,1)}(z,\bar{z}) = 0\) and consequently \(A = B = 0\).
Hence, the CR Ahlfors tensor must vanish in this case.
\end{remark} 
\subsection{Isometries}
The Heisenberg hypersurface \(\mathbb{H}^{2n-1}\) divides \(\mathbb{C}^n\) into
two half-spaces. The \emph{Siegel upper half-space} is
\begin{align*}
\Omega^{+}_{n} := \left\{ (z, w) \in \mathbb{C}^{n-1} \times \mathbb{C} : \operatorname{Im} w > |z|^2 \right\}.
\end{align*}
It is an unbounded strictly pseudoconvex domain in \(\mathbb{C}^n\). The Bergman
kernel of \(\Omega^{+}\) is explicit, see, for example, Krantz \cite{kr}. The 
Bergman metric on \(\Omega^{+}\)
is the K\"ahler metric given by the K\"ahler form
\begin{align*}
\omega_{\Omega^{+}}
= -ni\, \partial\bar{\partial} \log\left( \operatorname{Im} w - |z|^2 \right).
\end{align*}
Then, \(\Omega^{+}\) with this ``canonical'' metric is a model for the complex
hyperbolic space, which has constant negative holomorphic sectional curvature.

More generally, for each \(0<l<n-1\), the real hyperquadric \(\mathbb{H}^{2n-1}_l\)
of signature \(0<l<n-1\) divides
\(\mathbb{C}^{n}\) into two unbounded domains. Define the ``upper''
half-space to be
\begin{align*}
\Omega^{+}_{n,l}
:= \left\{ (z, w) \in \mathbb{C}^{n-1} \times \mathbb{C} 
: \operatorname{Im} w > \langle z, \overline{z}\rangle_l \right\}.
\end{align*}
The ``canonical'' pseudo-K\"ahler metric on $\Omega^{+}_l$
is given by
\begin{equation}
\omega_{\Omega^{+}_{n,l}} 
= -ni \partial\bar{\partial} \log\left( \operatorname{Im} w - \langle z, \overline{z}\rangle_l \right).
\end{equation}
As in the case of Siegel upper-half space, \(\omega_{\Omega^{+}_{n,l}}\) is a model for
the indefinite complex hyperbolic space, i.e., \(\omega_{\Omega^{+}_{n,l}}\) is a pseudo-K\"ahler metric having a constant negative holomorphic sectional curvature.

A one-sided neighborhood of \(\mathcal{X}_l^{2n+1}\) in \(\mathbb{C}^{n+1}\)
also possesses an interesting (pseudo) K\"ahler metric. We consider the ``upper'' domain
\begin{align*}
D^{+}_{n, l} 
= \left\{(z, \zeta, w) \in \mathbb{C}^{n+1} \mid (1 - |\zeta|^2) \Im  w - \langle z,\overline{z}\rangle_l - \Re \left(\bar{\zeta} zz^t\right)>0, \ 1-|\zeta|^2 > 0 \right\}
\end{align*}
and the ``canonical'' (pseudo) K\"ahler metric given by the fundamental 
form
\begin{equation}
\label{e:metriclieball}
\omega_{D^{+}_{n, l}} 
= -i \partial\bar{\partial} \log\left((1 - |\zeta|^2) \Im  w - \langle z,\overline{z}\rangle_l - \Re \left(\bar{\zeta} zz^t\right)\right).
\end{equation}
When \(l = 0\), this metric is locally isometric (up to a dimensional constant) to
the Bergman metric on the classical domain of type IV, the Lie ball.
When \(l>0\), Eq. \eqref{e:metriclieball} gives an interesting
pseudo-K\"ahler metric with the Ricci form satisfying
\[
	\mathrm{Ric} = -n \,\omega_{D^{+}_{n, l}}.
\]
Thus, it is (pseudo) K\"ahler--Einstein metric, and is locally isometric to a metric on the generalized Lie ball.

Similarly to Reiter--Son \cite{ReiSon2024} for the case \(l=0\), we have a close
relation between an isometry of a one-sided neighborhood of the canonical metrics and
the vanishing of the geometric rank of the map on the real hypersurface.

\begin{proposition}
Let \(p\in \mathbb{H}_l^{2n-1}\) and let \(U\) be an open neighborhood of \(p\) in 
\(\mathbb{C}^n\). Assume that \(U\cap \mathbb{H}_l^{2n-1}\) is open connected
and \(H \colon U \cap \mathbb{H}_l^{2n-1} \to \mathcal{X}_l^{2n+1}\) is a smooth 
CR map. If \(H\) extends to an isometry from \(U\cap \Omega^{+}_l\) into \(D^{+}_{n+1,l}\) with respect to the pseudo-K\"ahler metrics described above,
then \(H\) has vanishing geometric rank along \(U\cap \mathbb{H}_l^{2n-1}\).
\end{proposition}
\begin{remark}\rm 
By the Lewy extension theorem (see Baouendi--Ebenfelt--Rothschild \cite{ber99}), if \(0<l<n-1\), then \(H\) always extends holomorphically
to both sides of \(\mathbb{H}^{2n-1}_l\cap U\).
\end{remark} 
\section{The classification of CR maps from $\mathbb{H}^{2n-1}_l$ to $\mathcal{X}^{2n+1}_l$}
\label{sect:3}
In this section we will prove Theorem \ref{thrmHn1Xn2}. Let $U$ be an open subset 
of $\mathbb{H}^{2n-1}_l$ and $H \colon U\to\mathcal{X}^{2n+1}_l\subset\mathbb{C}^{n+1}$
be a $C^2$-smooth function. Note that the case \(l=0\) is known. 
Therefore, in the following \(1\leq l \leq n/2\).

By normalization, $F$ is equivalent to a map of the form 
given in Proposition~\ref{Claim normal1}. 
We therefore assume that $F$ is a holomorphic or formal map and already of this form. Similarly
to Reiter--Son \cite{ReiSon2022, ReiSon2024}, we first determine the map along the
first Segre set.

\begin{claim}
\label{claim31}
With the assumptions and notations above, it holds that
\begin{equation}
g(z,0)=0, \quad f(z,0)=\dfrac{2z}{1+\sqrt{1-4i\overline{\lambda}z z^t}},
\end{equation}
where as above
\(f = (f_1, \dots, f_{n-1}), z = (z_1, \dots, z_{n-1})\), and \(zz^t = z_1^2 + z_2 ^2 + \cdots + z_{n-1}^2\).
\end{claim}
\begin{proof}
Setting $(\overline{z}, \overline{w}) = (0, 0)$ in the mapping equation, we obtain $g(z,0)=0$. 

Define the following differential operators
\[
L_j :=\frac{\partial}{\partial\overline{z}_j}
-2i\epsilon_{j}z_j\frac{\partial}{\partial\overline{w}},
\quad j = 1, 2, \dots, n-1, \ \epsilon_j = \sign(j-l-1/2),
\]
where \(\sign(x)\) is the signum function. Observe that if 
\(\varphi(z, w)\) is holomorphic, then
\[
	L_j(\varphi(z, \bar{w} + 2i \langle z, \bar{z}\rangle_l)) = 0.
\] 
Thus, applying $L_j$ to the mapping equation \eqref{e:mapeq} yields
\begin{equation}
\label{e:ljmapeq}
(\phi \bar{\phi} -1) (L_j \bar{g}) - 2i \sum_{k=1}^{n-1} \epsilon_{k} f_k (L_j \bar{f}_k) - i F (L_j \bar{\phi}) - i \phi(L_j \bar{F}) = 0,
\end{equation}
where the ``bared'' functions are evaluated at \((\bar{z}, \bar{w})\)
and the ``unbared'' functions are evaluated at \((z, \bar{w} + 2i \langle z, \bar{z} \rangle_l)\).

From the partial normal form \eqref{e:normalform}, when evaluating at $(\overline{z}, \overline{w}) = (0, 0)$, it holds that
\begin{align}
(L_j \bar{g})(z, 0, 0) & = -2i \epsilon_j z_j, \\
(L_j \bar{f}_k)(z, 0, 0) & = \delta_{jk}, \\
(L_j \bar{\phi})(z, 0, 0) & = -2i \bar{\lambda} \epsilon_{j} z_j, \\
(L_j\bar{F})(z, 0, 0) & = 0.
\end{align}
Thus, evaluating \eqref{e:ljmapeq} at \(\bar{z} = 0\) and \(\bar{w} = 0\),
we obtain:
\begin{equation}\label{e:00}
z_j-f_j(z,0)+iz_j\overline{\lambda} F(z, 0)=0, \quad j = 1, 2, \dots, n-1.
\end{equation}
where, as above, \(F = f f^t\). Thus,
\begin{equation}\label{e:01}
 f_j(z, 0) = z_ j(1 + i \bar{\lambda} F(z, 0)).
\end{equation}
Taking the squares of both sides of \eqref{e:01} and summing over \(j\), we obtain a quadratic equation for 
\(F(z, 0)\). Namely,
\begin{equation}
\label{e:02}
F(z, 0) = z z^t(1 + i \lambda F(z, 0))^2.
\end{equation}
From \eqref{e:02}, we can solve for \(F(z, 0)\) and then solve the
system of equations \eqref{e:00} to obtain
\begin{equation}\label{standardf}
 f_j(z,0)=\dfrac{2z_j}{1+\sqrt{1-4i\overline{\lambda} z z^t}}, \quad j = 1, 2, \dots, n-1 .
\end{equation}
We complete the proof.
\end{proof}

\begin{claim}\label{claim32}
With the assumptions and notations as above, it holds that
\begin{align*}
\frac{\partial g}{\partial w}(z, 0)=\frac{2}{1+\sqrt{1-4i\overline{\lambda}zz^t}}.
\end{align*}
\end{claim}
\begin{proof}
Differentiating the mapping equation \eqref{e:mapeq} with respect to $\overline{w}$ and setting $(\overline{z}, \overline{w}) = (0, 0)$ yield
\[
-\overline{\lambda}\left(\sum_{j=1}^{n-1}f_j(z,0)^2\right)
-i\left(-1+\dfrac{\partial g}{\partial w}(z,0)\right)=0.
\]
Using Proposition \ref{claim31} and solving for $\dfrac{\partial g}{\partial w}(z,0)$ yield
\begin{equation}\label{standardg}
 \frac{\partial g}{\partial w}(z,0)
 =\dfrac{2}{1+\sqrt{1-4i\overline{\lambda}zz^t}}.
\end{equation}
The proof is complete.
\end{proof}

From here, for simplicity, we divide our consideration into several cases depending on the value of $\lambda$ and $n$.\\

\textbf{Case 1:} $\lambda=0$.\\

In this case, from Propositions \ref{claim31} and \ref{claim32}, it holds that
\[
f(z, 0) = z, \quad g(z, 0) = 0, \quad g_w(z, 0) = 1.
\]

\textbf{Subcase 1.1:} $n\geq 4$.

\begin{claim}
\label{claim33}
With the assumptions and notations as above, it holds that
\begin{align*}
	\sigma=0,v=0,\mu=0, \textit{ and } c_{jk}=0 \text{ for all } 1\leq j<k\leq n-1.
\end{align*}
\end{claim}
\begin{proof}
Applying $L_k$ to \eqref{e:ljmapeq}, we obtain
\begin{align}
\label{e:lkljmapeq}
\phi(L_k\bar{\phi})(L_j\bar{g}) + (\phi \bar{\phi} - 1) (L_kL_j \bar{g}) 
-2i \left\langle f , L_k L_j \bar{f} \right\rangle_l 
- i F (L_k L_j \bar{\phi}) - i \phi(L_k L_j \bar{F}) = 0.
\end{align}
Setting $(\overline{z}, \overline{w}) = (0, 0)$, we have
\begin{align*}
 L_k L_j \bar{F} & = 2 \delta_{jk}, \\
 L_k L_j \bar{g} & = 0, \\
 L_k L_j \bar{\phi} & = \overline{a}_{jk} - 2i \epsilon_{j}z_j \bar{\mu}_k
 - 2i \epsilon_{k} z_k \overline{\mu}_j - 8  \epsilon_{j}\epsilon_{k}z_j z_k \sigma, \\
 L_k L_j  \bar{f}_q & = -\epsilon_{j} z_j \overline{a}_{kq} - \epsilon_{k} z_k \overline{a}_{jq} - 8 \epsilon_{j}\epsilon_{k} z_j z_k v_q.
\end{align*}
Substituting these into \eqref{e:lkljmapeq}, with \(k\ne j\) we obtain various polynomial equations of~$z$. Equating the coefficients of both sides, we have 
$\sigma=0, v=0, \mu=0$ and $c_{jk}=0$ for all $1\leq j < k\leq n-1$.

Next, substituting into \eqref{e:lkljmapeq} when \(j = k\), we find that
\[
	\phi(z, 0) = 0.
\]
The proof is complete.
\end{proof}
\begin{remark}
\rm 
We mentioned that it is also possible to deduce that $c_{jk} = 0$ by invoking the well-known Huang's Lemma \cite{hu99}. On the other hand, the proof above
breaks down for \(n=3\). In this case (which shall be treated below), we have only two differential operators
\(L_1\) and \(L_2\) leading to only two equations that are not enough to conclude \(c_{jk} = 0\).
\end{remark}

\begin{claim}
\label{claim34}
With the assumptions and notations as above,
\begin{align}
	\phi(z,w) = 0.
\end{align}
\end{claim}

\begin{remark}\rm 
Once we prove \(\phi(z,w) = 0\), our map reduces to \(H = (f, 0, g)\).
From the mapping equation, it follows that  the map
\(\tilde{H} = (f, g)\) is a map between hyperquadrics of the same dimension.
By the well-known 2-jet determination result in Chern--Moser \cite{cm74}, we can infer from the 
partial normal form \eqref{e:normalform} that \(\tilde{H} = (z, w)\) and hence \(H = (z, 0, w) = \ell^{(n)}(z, w)\), as desired.
\end{remark}
\begin{proof}[Sketch of proof of Proposition \ref{claim34}]
From the formula for the map \(H\) along the first Segre set: \(H(z, 0) = (z, 0, 0)\), we can use a ``reflection principle'' argument as in Reiter \cite{Reiter16a}
and Reiter--Son \cite{ReiSon2022, ReiSon2024} (cf. Baouendi--Ebenfelt--Rothschild \cite{BER}) to produce
several holomorphic equations for components \(H\), which,
in turn, produce desired formula for \(H\).

Namely, to compute $\phi(z,w)$, first we multiply the mapping
equation with $2$ and set $\overline{w}=0$ to obtain 
\begin{equation}
\label{e:I1}
	I = I(z, w, H(z, 2i \langle z, \bar{z}\rangle_l), \bar{H}(\bar{z}, 0)) = 0.
\end{equation}
Substituting  \(\bar{H}(\bar{z}, 0) = (\bar{z}, 0, 0)\) yields an equation of the form
\begin{equation}
\label{e:I2}
	\tilde{I} = \tilde{I} (z, w, H(z, 2i \langle z, \bar{z}\rangle_l), \bar{z}) = 0.
\end{equation}
For each $1\leq j\leq n-1$, multiplying the equation above with $4z_j^2$ and setting
\[
\overline{z}_j
=\dfrac{w-2i\sum_{1\leq k\leq n-1, k\ne j} \epsilon_{k} z_k\overline{z}_k}{2i\epsilon_jz_j},
\]
we obtain an equation of the form
\[
	\mathcal{I}_j(z, w, H(z, w), \bar{z}_1, \dots , \hat{\bar{z}}_j, \dots \bar{z}_{n-1}) = 0,
\]
where the variable \(\bar{z}_j\) does not appear. Thus, we obtain $n-1$ such 
equations. Two of them lead to \(\phi(z, w) = 0\). Indeed, taking the derivative of $\mathcal{I}_1$ with respect to $\bar{z}_2$, setting $ \bar{z} = 0$, and multiplying by $z_2$, we obtain the first equation. Taking the derivative of $\mathcal{I}_2$ with respect to $\overline{z}_1$, setting $\overline{z} = 0$ and multiplying by $z_1$, we obtain the second equation. Taking the difference of the two equations give $\phi(z,w)=0$.
The proof is complete.
\end{proof}

Thus, in the case $n\geq 4$, every map \(H\) in the partial normal form 
\eqref{e:normalform} with \(\lambda = 0\) must be of the form $\ell^{(n)}(z,w)=(z, 0, w)$.\\

\textbf{Subcase 1.2}: $\lambda=0$ and $n=3$.\\

In this case, our consideration is somewhat similar
to Reiter--Son \cite{ReiSon2024}. For example, the following is almost the same as 
\cite[Lemma~4.3]{ReiSon2024}

\begin{claim}\label{claim37}
With the assumptions and notations as above,
 $\sigma=0,v=0$, and $\mu=0$.
\end{claim}
\begin{proof}
Applying $L_1$ and $L_2$ consecutively to the mapping equation \eqref{e:mapeq}, 
setting $\overline{z}=\overline{w}=0$, and applying
Propositions \ref{claim31} and \ref{claim32},
 we obtain a polynomial of $z$:
\begin{equation}\label{equamusigup}
 8\overline{v}_1 z_1^2z_2-8\overline{v}_2 z_1z_2^2-(z_1^2+z_2^2)(-iz_2\overline{\mu}_1 + iz_1\overline{\mu}_2 + 4z_1z_2\overline{\sigma})=0.
\end{equation}
Now equating the coefficients of both sides,
we have $\sigma=0, v=0$ and $\mu=0$. This completes the proof.
\end{proof}

\begin{claim}
\label{claim38}
Assume that \(H\) has a normal form \eqref{e:normalform} with \(\lambda = 0\). Then
\begin{equation}
 \phi(z, 0) = z A z^t.
\end{equation}
More precisely, if
\[
 A = \begin{pmatrix} 
 \alpha & i\beta \\ i\beta & -\alpha 
 \end{pmatrix},
\]
where \(\alpha, \beta \in \mathbb{R}\), then $\phi(z,0)=\alpha(z_1^2-z_2^2)+2i\beta z_1z_2$.
\end{claim}
\begin{proof}
Applying $L_1$ twice to the mapping equation and setting $\overline{z}=\overline{w}=0$ we obtain that
\begin{align*}
 \alpha(z_1^2-z_2^2)+2i\beta z_1z_2-\phi(z,0)=0,
\end{align*}
from which we complete the proof.
\end{proof}
From the formula for the map \(H\) along the first Segre set as above,
we can use a ``reflection principle'' argument as above to produce
several holomorphic equations for components \(f, \phi\) and \(g\) of the map.

\begin{claim}\label{claim39}
The following equations hold in a neighborhood of the origin of~$\mathbb{C}^3$:
\begin{equation}\label{eq49}
\alpha w^2\Phi-4wz_1f_1(z,w)+4z_1^2g(z,w)+iw^2\phi(z,w)=0
\end{equation}
\begin{equation}\label{eq48}
\alpha w^2\Phi+4wz_2f_2(z,w)-4z_2^2g(z,w)-iw^2\phi(z,w)=0
\end{equation}
\begin{equation}\label{eq411}
w(z_2\alpha-iz_1\beta)\Phi-2z_1(z_2f_1(z,w)-z_1f_2(z,w))+iwz_2\phi(z,w)=0
\end{equation}
\begin{equation}\label{eq410}
w(z_1\alpha+iz_2\beta)\Phi-2z_2(z_2f_1(z,w)-z_1f_2(z,w))-iwz_1\phi(z,w)=0
\end{equation}
\begin{equation}\label{eq412}
(\alpha(z_1^2-z_2^2)+2i\beta z_1z_2)\Phi-i(z_1^2+z_2^2)\phi(z,w)=0
\end{equation}
with $\Phi=g(z,w)\phi(z,w)+i(f_1(z,w)^2+f_2(z,w)^2)$.
\end{claim}
\begin{proof}[Sketch of the proof]
Applying $\overline{w}=0$ to the mapping equation and substituting 
$\overline{z_1}=0, \overline{z_2}=\dfrac{w}{2iz_2}$ we obtain (\ref{eq48}).
Similarly, applying $\overline{w}=0$ to the mapping equation and substituting
$\overline{z_2}=0, \overline{z_1}=\dfrac{-w}{2iz_1}$ we obtain (\ref{eq49}).
Now we define the function $J$ by multiplying the mapping equation with $2$ and
applying $\overline{w}=0$, we obtain an identity of the form
\begin{equation}
\label{e:funcR}
	\mathcal{R}(z, w, H(z, w), \overline{H}(\bar{z}, 0)) = 0,
\end{equation}
where \(\mathcal{R}\) is explicit and polynomial in its arguments. 
We won't reproduce the explicit form of \(\mathcal{R}\) here for simplicity.
From this and the identity for \(\overline{H}(\bar{z}, 0)\) obtained
above, we get an identity of the form
\begin{align*}
	\mathcal{R}_1(z, w, H(z, w), \bar{z}) = 0,
\end{align*}
where, as above, \(\mathcal{R}_1\) is also explicit and polynomial.
Next, substituting $\overline{z_1}=\dfrac{w-2iz_2\overline{z_2}}{-2iz_1}$
into this, we obtain an identity of the form
\begin{align*}
	\mathcal{R}_2(z, w, H(z, w), \bar{z}_2) = 0,
\end{align*}
with \(\mathcal{R}_2\) is polynomial in its arguments (after clearing some denominator).
Differentiating this with respect to \(\bar{z}_2\) and setting 
\(\bar{z}_2=0\), we obtain \eqref{eq411}. 

By the same procedure as above with
the roles of $z_1$ and $z_2$ exchanged, we obtain
\eqref{eq410}.

Finally, multiplying (\ref{eq411}) by $z_2$, multiplying (\ref{eq410}) by $z_1$ and 
taking the difference of the resulting equations, we obtain (\ref{eq412}).
The proof is complete.
\end{proof}

\begin{claim}\label{claim310}
Assume that \(\lambda = 0\). It holds that
\begin{align}
 f_1(z,w) & =\dfrac{z_1}{w}g(z,w)+\dfrac{w(z_1\alpha+iz_2\beta)}{2(z_1^2+z_2^2)}\Phi, \\ f_2(z,w) & =\dfrac{z_2}{w}g(z,w)-\dfrac{w(z_2\alpha-iz_1\beta)}{2(z_1^2+z_2^2)}\Phi, \\
 \phi(z,w) & =\dfrac{-i(\alpha(z_1^2-z_2^2)+2i\beta z_1z_2)}{z_1^2+z_2^2}\Phi.
\end{align}
\end{claim}
\begin{proof}
We can rewrite the equations in Proposition~\ref{claim39} as a system of linear equations of $5$ variables $f_1,f_2,\phi,g$, and $\Phi$. Solving the system of equations (\ref{eq49}), (\ref{eq48}), (\ref{eq412}) we have the desired formulas above.
\end{proof}

In order to fully determine \(H(z, w)\), we need more holomorphic equations
for its component. To this end, we shall determine \(H_w\) along the first 
Segre set in the next proposition.
\begin{claim}
\label{claim311}
Assume that \(\lambda = 0\). It holds that
\begin{align}
\dfrac{\partial f_1}{\partial w}(z,0) &=\dfrac{i}{2}(z_1\alpha+iz_2\beta),\\
\dfrac{\partial f_2}{\partial w}(z,0) &=\dfrac{i}{2}(-z_2\alpha+i z_1\beta),\\
\dfrac{\partial \phi}{\partial w}(z,0)&=0.
\end{align}
\end{claim}
\begin{proof}
Applying $L_1$ and $T$ to the mapping equation consecutively at $\overline{z}=\overline{w}=0$, we obtain:
$$\dfrac{\partial f_1}{\partial w}(z,0)=\dfrac{i}{2}(z_1\alpha+iz_2\beta).$$
Applying $L_2$ and $T$ to the mapping equation consecutively at $\overline{z}=\overline{w}=0$, we obtain:
$$\dfrac{\partial f_2}{\partial w}(z,0)=\dfrac{i}{2}(-z_2\alpha+iz_1\beta).$$
Differentiating (\ref{eq412}) with respect to $w$ and setting $w=0$, we obtain:
$$\dfrac{\partial \phi}{\partial w}(z,0)=0.$$
The proof for Proposition~\ref{claim311} is completed.
\end{proof}

From the formulas for \(H\) and \(H_w\) along the first Segre set as above,
we can produce another several holomorphic equations for components of the map.
One of such equations is as follows.

\begin{claim}\label{claim312}
Assume that \(\lambda = 0\). It holds that
\begin{equation}\label{eq419}
w\alpha\Phi-z_1(2+iw\alpha)f_1(z,w)+wz_1\beta f_2(z,w)+iw\phi(z,w)+2z_1^2=0.
\end{equation}
\end{claim}
\begin{proof}
Evaluating \eqref{e:ljmapeq} with \(j=1\) at 
$\overline{z}_1=\dfrac{iw}{2z_1}, \overline{z}_2=0$ and $\overline{w}=0$,
we obtain \eqref{eq419}.
The proof is complete.
\end{proof}

\begin{theorem}
\label{thm:2params}
Let \(H\) be a holomorphic map in a neighborhood of the origin sending the germ at the origin of hyperquadric
\(\mathbb{H}_1^5\) into \(\mathcal{X}_1^7\). Assume that \(H\) has the partial normal form \eqref{e:normalform}
with \(\lambda = 0\). Then \(H\) is given by
\begin{align*}
H(z,w) = H_A(z,w)
= \left(\frac{4z + 2i w z A z^t}{4 + |A| w^2},
\frac{4 z A z^t}{4 + |A| w^2 }, \frac{4w}{4 + |A| w^2}\right),
\end{align*}
where 
\begin{equation}
\label{e:Amatrix}
A = \begin{pmatrix} \alpha & i\beta \\ i\beta & -\alpha \end{pmatrix},
\end{equation}
with \(\alpha, \beta \in \mathbb{R}\), so that \(|A| = -\alpha^2 + \beta^2\).
Conversely, each matrix \(A\) of the form \eqref{e:Amatrix} gives rise to a 
holomorphic map
sending the hyperquadric into \(\mathcal{X}_1^7\).
\end{theorem}
Thus, at this point we completely determine all rational holomorphic maps sending the 5-dimensional
hyperquadric of signature 1 into \(\mathcal{X}_1^{7}\) CR transversally.
\begin{proof}
Applying Proposition~\ref{claim310} to (\ref{eq419}) and taking the numerator,
we obtain an equation for $g(z,w)$ and $\Phi(z,w)$ as follows:
\begin{align}\label{eq420}
	0=4(zz^t)^2\alpha g(z,w)^2-4z_1zz^t(z_1(w\alpha -2i )+iwz_2\beta)g(z,w) \notag \\
	+w(-8iz_1^4+4wz_1z_2\beta\Phi
	+w\alpha\Phi(4iz_2^2+w^2(\alpha^2-\beta^2)\Phi)-2z_1^2(4iz_2^2+w^2(\alpha^2-\beta^2)\Phi)).
\end{align}

Applying Proposition~\ref{claim310} to the equation $\Phi-(g(z,w)\phi(z,w)+i(f_1(z,w)^2+f_2(z,w)^2))=0$ and taking the numerator, we obtain an equation of $g(z,w)$ and $\Phi(z,w)$:
\begin{equation}\label{eq421}
 4(z_1^2+z_2^2)^2g(z,w)^2+w^2\Phi(4i(z_1^2+z_2^2)+w^2(\alpha^2-\beta^2)\Phi)=0.
\end{equation}
Solving the equation from (\ref{eq420}) and (\ref{eq421}) we obtain
\begin{align*}
g(z,w) & = \dfrac{4w}{4-w^2(\alpha^2-\beta^2)},\\
\Phi(z,w) & =\dfrac{4i(z_1^2+z_2^2)}{4-w^2(\alpha^2-\beta^2)}.
\end{align*}
Substituting into Proposition~\ref{claim310} we have the formulas for $f_1(z,w)$, 
$f_2(z,w)$ and $\phi(z,w)$.
\begin{align*}
f_1(z,w) & = \dfrac{z_1(4+2iw\alpha)-2wz_2\beta}{4-w^2(\alpha^2-\beta^2)}, \\
f_2(z,w) & = \dfrac{z_2(4-2iw\alpha)-2wz_1\beta}{4-w^2(\alpha^2-\beta^2)}, \\
\phi(z,w)& = \dfrac{4(\alpha(z_1^2-z_2^2)+2iz_1z_2\beta)}{4-w^2(\alpha^2-\beta^2)}.
\end{align*}
Thus, \(H = H_A\), as desired. Conversely, it can be checked directly that each
map \(H_A\) is transversal to \(\mathcal{X}_1^7\) and sends \(\mathbb{H}_1^5 \setminus \mathrm{Sing}(H_A)\)
into \(\mathcal{X}_1^7\). The proof is complete.
\end{proof}
In the sequel, we shall show that this 2-parameter family of maps reduces to 
five equivalence classes, represented by \(H_{A_j}\), \(j=1,2,\dots, 5\), where
\[
A_1 = \begin{pmatrix} 0 & 0 \\ 0 & 0 \end{pmatrix},\ A_2 = \begin{pmatrix} 2 & 2i \\ 2i & -2 \end{pmatrix}, 
\ A_3 = \begin{pmatrix} -2 & -2i \\ -2i & 2 \end{pmatrix},
\]
\[
\ A_4 = \begin{pmatrix} 2 & 0 \\ 0 & -2 \end{pmatrix},
\ A_5 = \begin{pmatrix} 0 & 2i \\ 2i & 0 \end{pmatrix}.
\]
\begin{enumerate}

\item 
If $\alpha=\beta=0$, we obtain the map $H_0(z,w) = (z, 0, w) = H_{A_1}(z, w)$.

\item 
If $\alpha=\beta>0$, we write \(\beta = 2s^2\) with \(s>0\) and consider the following automorphisms $\Psi_2(z,w)\in \Aut(\mathbb{H}^5_1)$
and $\gamma_2\in \Aut(\mathcal{X}^{7}_1)$:
\[
\Psi_2=\left(sz_1, sz_2, s^2 w\right)\]
and
\[
\gamma_2=\left(s z_1, s z_2,\zeta,s^2 w\right).\]
Clearly, $\gamma_2\circ H\circ \Psi_2^{-1} = H_{A_2}$.

\item 
If $-\alpha = \beta > 0$, then it can be shown similarly as above that
$H$ is equivalent to $H_{A_2}$. We omit the details.

\item
If $\alpha=\beta<0$, we write \(\beta = - 2s^2\) with \(s>0\)
and consider the following automorphisms $\Psi_3(z,w)\in \Aut(\mathbb{H}^3_1)$
and $\gamma_3\in \Aut(\mathcal{X}^{7}_1)$:
\[
\Psi_3=\left(-isz_2,isz_1,s^2w\right)
\]
and
\[\gamma_3=\left(-isz_2,isz_1,\zeta,s^2w\right).\]
Composing $\gamma_3\circ H\circ \Psi_3^{-1}$ gives $H_{A_3}$.

\item 
If $-\alpha = \beta < 0$, then it can be shown similarly as above that
$H$ is equivalent to $H_{A_3}$. We omit the details.

\item
If $\alpha^2-\beta^2>0$, we put $\alpha = 2r^2 \cosh(s)$ and $\beta = 2r^2\sinh(s)$, \(r>0\),
\[
 B = \begin{pmatrix} \cosh(s/2) & i \sinh(s/2) \\ -i \sinh(s/2) & \cosh(s/2)
 \end{pmatrix}, \quad B^{-1} = B^t
\]
and consider the following automorphisms $\Psi_4(z,w)\in \Aut(\mathbb{H}^5_1)$
and $\gamma_4\in \Aut(\mathcal{X}^7_1)$
\begin{align*}
 \Psi_4(z, w) & = \left(r zB, r^2 w \right), \\
 \gamma_4(z, w) & =\left(rz B, \zeta, r^2 w \right).
\end{align*}
Now composing $\gamma_4\circ H\circ \Psi_4^{-1}$ gives us the mapping $H_{A_4}$.

\item
If $\alpha^2-\beta^2<0$, we put $\alpha = 2r^2\sinh(s)$ and
$\beta = 2r^2\cosh(s)$. Now we consider the maps
$\Psi_5(z,w)\in \Aut(\mathbb{H}^3_1)$ and $\gamma_5\in \Aut(\mathcal{X}^{7}_1)$ which have the following formulas:
$$\Psi_5=\left(r\left(-z_1 \cosh\left(\dfrac{s}{2}\right)-iz_2\sinh\left(\dfrac{s}{2}\right)\right),r\left(iz_1\sinh\left(\dfrac{s}{2}\right)-z_2\cosh\left(\dfrac{s}{2}\right)\right), r^2 w\right)$$
and
$$\gamma_5=\left(r\left(-z_1\cosh\left(\dfrac{s}{2}\right)-iz_2\sinh\left(\dfrac{s}{2}\right)\right),r\left(iz_1\sinh\left(\dfrac{s}{2}\right)-z_2\cosh\left(\dfrac{s}{2}\right)\right),\zeta,r^2 w\right).$$
Now composing $\gamma_5\circ H\circ \Psi_5^{-1}$ gives us the mapping $H_{A_5}$.
\end{enumerate}

\begin{proposition}
For \(j\ne k\), the germs at the origin of \(H_{A_j}\) and \(H_{A_k}\) are inequivalent.
\end{proposition}
\begin{proof}
The geometric rank of \(H_{A}\) at the origin is equal to the rank of the matrix
\(A\). Thus,
\(H_{0}\) has vanishing geometric rank, \(H_{A_2}\) and \(H_{A_3}\)
have geometric 
rank 1, while \(H_{A_4}\) and \(H_{A_5}\) have rank two at the origin.

To distinguish two maps with the same geometric rank, we can look at the eigenvalues of
the CR Ahlfors tensors of the maps at the origin. By direct computations, the component of the CR Ahlfors tensor
in the local holomorphic frame \(\{\overline{L}_1, \overline{L}_2\}\)
is given by
\[
 \mathcal{A}(H_{A})\Big|_0 = \begin{pmatrix} -\alpha & -i\beta \\ i\beta & -\alpha \end{pmatrix}.
\]
From this, the inequivalences of \(H_{A_j}\) for different \(j\) are evident.
\end{proof}

\textbf{Case 2:} $\lambda\ne 0$.

In this case, the map must be irrational. 

\begin{theorem}
\label{claim317}
Assume that the germ \(H\) is of the form \eqref{e:normalform} with \(\lambda \ne 0\).
Then
\begin{equation}
 H(z,w)
 =\dfrac{2(z,\lambda w,w)}{1+\sqrt{1-4i\overline{\lambda}\left(z z^t-i\lambda w^2\right)}}.
\end{equation}
\end{theorem}
\begin{remark}
\rm 
The partial normal form \eqref{e:normalform} also determines the map uniquely
in this case. It is interesting to point out that we also get a two-parameter analytic family of CR maps
containing the linear map. Each map in the family is either equivalent to
the linear map or the irrational map depending on whether \(\lambda = 0\) or \(\lambda \ne 0\).
\end{remark}

In what follows, we will prove Theorem \ref{claim317} via several propositions.

\begin{claim}\label{claim313}
If \(\lambda \ne 0\), then 
\[
\sigma=0, v=0,\mu=0 \text{ and } c_{ij}=0
\]
for all $1\leq i<j\leq n-1$.
\end{claim}
\begin{proof}[Sketch of the proof]
The idea of proof is similar to that of \cite[Lemma 4.9]{ReiSon2024}. We therefore only sketch the proof.
For each \(j\),
applying $L_j$ twice to the mapping equation we obtain \(n-1\) equations of the form \eqref{e:lkljmapeq} (with \(k=j\)).
Taking the difference of the first equation (i.e, \(j=1\)) times $(i+4\overline{\lambda}z_2^2)$ and the second equation (i.e, \(j=2\)) times 
$(i+4\overline{\lambda_1}z_1^2)$ we obtain an equation of the form:
\begin{align*}
M(z) + N(z) \sqrt{1-4i\overline{\lambda}z z^t}=0,
\end{align*}
where $M(z)$ and $N(z)$ are polynomials in \(z\). These
calculations are quite lengthy and tedious, but can be done 
quickly with help of a computer algebra system.

Next, as $\lambda \ne 0$, we have 
$$M(z) = N(z) =0.$$
Now equating the coefficients both sides of the equation 
gives us the desired claim.
\end{proof}
\begin{claim}\label{claim314}
 $\phi(z,0)=0$ and $\dfrac{\partial f}{\partial w}(z,0)=0$.
\end{claim}
\begin{proof}
Applying $L_2$ twice to the mapping equation, evaluating at 
$\overline{z}=\overline{w}=0$, and combining with Propositions \ref{claim32} 
and \ref{claim313}, we have 
\[
(-1+4i\overline{\lambda}z_2^2)\phi(z,0)=0.
\]
Thus, $\phi(z,0)=0$. Next, applying $T$ followed by $L_j$ to the mapping
equation \eqref{e:mapeq},
evaluating at $\overline{z}=\overline{w}=0$, and combining
with $\phi(z,0)=0$ we obtain $n-1$ equations:
\[
\left(1-4i\overline{\lambda}z_j^2+\sqrt{1-4i\overline{\lambda}z z^t}\right)
\dfrac{\partial f_j}{\partial w}(z,0)
-4i\overline{\lambda} z_j\left(\sum_{1\leq i\ne j\leq n-1}z_i\dfrac{\partial f_i}{\partial w}(z,0)\right)=0.
\]
Solving this system of equations we obtain that $\dfrac{\partial f}{\partial w}(z,0)=0$.
\end{proof}

Similarly to \textit{Case 1}, by multiplying the mapping equation
\eqref{e:mapeq} with $2$, setting $\overline{w}=0$, and substituting the formulas for \(\overline{H}(\bar{z}, 0)\), we obtain an equation of the form
\begin{equation}
\label{e:I3}
	J := J (z, w, H(z, 2i \langle z, \bar{z}\rangle_l), \bar{z}) = 0,
\end{equation}
Now for each $1\leq j\leq n-1$,  by multiplying $J$ with $z_j^2$, setting $$\overline{z}_j=\frac{w-2i\sum_{1\leq k\leq n-1, k\ne j}\epsilon_k z_k\overline{z}_k}{2i \epsilon_j z_j},$$
and taking the numerators, we obtain an equation of the form
\begin{align*}
\mathcal{J}_j (z, w, H(z, w) ) = 0,
\end{align*}
for \(j = 1, 2, \dots, n-1\). Again, the explicit formula for \(\mathcal{J}\)
is quite complicated and is not provided here. But it can be 
computed quickly by using a computer algebra system.

\begin{claim}\label{claim315}
\begin{align}
\label{e:3.18a}
	\dfrac{\partial \phi}{\partial w}(z,0) &= \dfrac{2\lambda}{1+\sqrt{1-4i\overline{\lambda}zz^t}}, \\
	\dfrac{\partial g}{\partial w}(z,0) &= \dfrac{2}{1+\sqrt{1-4i\overline{\lambda}zz^t}}.
\label{e:3.18b}
\end{align}
\end{claim}
\begin{proof}
Taking the derivative of $\mathcal{J}_1$ with respect to $w$ and applying $\overline{z}=0$ and $w=\overline{w}=0$, we have:
\[
4iz_1^2\left(\dfrac{2}{1+\sqrt{1-4i\overline{\lambda}zz^t}}-\dfrac{\partial g}{\partial w}(z,0)\right)=0.
\]
From this, we obtain \eqref{e:3.18b}.
Taking the derivative of $\mathcal{J}_1$ with respect to $w$ and applying $w=0$, we obtain \eqref{e:3.18a} %
The proof is complete.
\end{proof}

Now we will divide into $2$ smaller cases depending on the value of $n$.\\

\textbf{Subcase 2.1:} $n\geq 4$. \\

For pairwise distinct indices $1\leq j, k, t\leq n-1$, we calculate the following
expression
\[
z_t\dfrac{\partial \mathcal{J}_j}{\partial\overline{z}_k} - \epsilon_k \epsilon_{t} z_k\dfrac{\partial \mathcal{J}_j}{\partial\overline{z}_t}
\]
at $\overline{z}_m=0$ for all $ m\ne j$, to obtain that:
\[
4z_j\left(z_j+\sqrt{z_j^2-i\lambda w^2}\right)(z_kf_t(z,w)-z_tf_k(z,w))=0.
\]
Thus we have $z_kf_t(z,w)=z_tf_k(z,w)$ for all $1\leq k\ne t\leq n-1$.
Therefore, to determine \(f(z, w)\), we only need to determine \(f_1(z, w)\).

\begin{claim}
\label{claim316}
With assumptions and notations as above,
\[
\phi(z,w)=\lambda g(z,w).
\]
\end{claim}
\begin{proof}
Taking the derivative of $\mathcal{J}_1$ with respect to $\overline{z_2}$,
applying $\overline{z}=0$, substitute $f_2(z,w)$ by $z_2f_1(z,w)/z_1$
and taking the numerator we have:
\begin{equation}\label{claim10eq1}
 w\lambda f_1(z,w)-z_1\lambda\left(1+\sqrt{1-\dfrac{iw^2\lambda}{z_1^2}}\right)g(z,w)+z_1\sqrt{1-\dfrac{iw^2\lambda}{z_1^2}}\phi(z,w)=0.
\end{equation}
Taking the derivative of $\mathcal{J}_2$ with respect to $\overline{z_1}$, applying $\overline{z}=0$, 
substitute $f_2(z,w)$ by $z_2f_1(z,w)/z_1$ and taking the numerator we have:
\begin{equation}\label{claim10eq2}
 w\lambda f_1(z,w)-z_1\lambda\left(1+\sqrt{1-\dfrac{iw^2\lambda}{z_2^2}}\right)g(z,w)+z_1\sqrt{1-\dfrac{iw^2\lambda}{z_2^2}}\phi(z,w)=0.
\end{equation}
Finally, subtracting the equation (\ref{claim10eq1}) to (\ref{claim10eq2}), we obtain that
$\lambda g(z,w)=\phi(z,w).$
\end{proof}

\begin{proof}[Proof of Theorem \ref{claim317} for the case $n\geq 4$]
Applying the operation $L_1$ to the mapping equation, setting
$\overline{w}=0, \overline{z}_1=\dfrac{iw}{2z_1}, \overline{z}_j=0$ for all $2\leq j\leq n-1$, and substituting $f_j(z,w)$ by $z_jf_1(z,w)/z_1$, $1\leq j\leq n-1$ and $g(z,w)$ by $wf_1(z,w)/z_1$, we obtain that
\begin{equation}\label{solveforf1case21}
 -z_1^2+z_1f_1(z,w)
 -i\left(z z^t-i\lambda w^2\right)\overline{\lambda}f_1(z,w)^2=0.
\end{equation}
Solving this equation and combining with the condition 
$f_1(z,0)=\dfrac{2z_1}{1+\sqrt{1-4i\overline{\lambda}z z^t}}$ gives us
$$f_1(z,w)=\dfrac{2z_1}{1+\sqrt{1-4i\overline{\lambda}\left(z z^t-i\lambda w^2\right)}}.$$
Now combining with Proposition \ref{claim315} we complete the proof for Theorem \ref{claim317}.
\end{proof}

\textbf{Subcase 2.2:} $n=3$.\\

Then applying $\overline{z}_2=0$ in $\mathcal{J}_1$ we obtain that:
\[
2iwz_1\left(1+\sqrt{1-\dfrac{iw^2\lambda}{z_1^2}}\right)f_1(z,w)-iz_1^2\left(1+\sqrt{1-\dfrac{iw^2\lambda}{z_1^2}}\right)^2g(z,w)+w^2\phi(z,w)=0.
\]
Applying $\overline{z_1}=0$ in $\mathcal{J}_2$ we obtain that:
\[
2iwz_2\left(1+\sqrt{1-\dfrac{iw^2\lambda}{z_2^2}}\right)f_2(z,w)-iz_2^2\left(1+\sqrt{1-\dfrac{iw^2\lambda}{z_2^2}}\right)^2g(z,w)+w^2\phi(z,w)=0.
\]
Solving for $f_1(z,w)$ and $f_2(z,w)$ we have:
\begin{equation}\label{solveforf1case22}
 f_1(z, w)=\dfrac{z_1^2\left(1+\sqrt{1-\dfrac{iw^2\lambda}{z_1^2}}\right)^2g(z,w)+iw^2\phi(z,w)}{2wz_1\left(1+\sqrt{1-\dfrac{iw^2\lambda}{z_1^2}}\right)}, 
\end{equation}
\begin{equation}\label{solveforf2case22}
 f_2(z, w)=\dfrac{z_2^2\left(1+\sqrt{1-\dfrac{iw^2\lambda}{z_2^2}}\right)^2g(z,w)+iw^2\phi(z,w)}{2wz_2\left(1+\sqrt{1-\dfrac{iw^2\lambda}{z_2^2}}\right)}.
\end{equation}
\begin{claim}\label{claim318}
 $\phi(z,w)=\lambda g(z,w)$.
\end{claim}
\begin{proof}
Taking the derivative of $\mathcal{J}_2$ with respect to $\overline{z}_1$, applying 
$\overline{z}_1 = 0$ and the formula (\ref{solveforf1case22}),
(\ref{solveforf2case22}) we obtain the desired formula.
\end{proof}

\begin{proof}[Proof of Theorem \ref{claim317} for $n=3$ and $l=1$]
Applying the operation $L_1$ to the mapping equation at $\overline{z}_1=\dfrac{iw}{2z_1},\overline{z}_2=\overline{w}=0$, using the formulas (\ref{solveforf1case22}), (\ref{solveforf2case22}) and Proposition~\ref{claim318} we obtain that
\[
w^2-wg(z,w)+(i(z_1^2+z_2^2-i\lambda w^2)\overline{\lambda}g(z,w)^2=0.
\]
Solving this equation we obtain that:
\begin{equation}\label{gcase22}
 g(z,w)=\dfrac{2w}{1+\sqrt{1-4i\overline{\lambda}\left(z_1^2+z_2^2-i\lambda w^2\right)}}.
\end{equation}
Using (\ref{solveforf1case22}), (\ref{solveforf2case22}) and Proposition~\ref{claim318} we have:
\begin{equation}\label{phicase22}
 \phi(z,w)=\dfrac{2\lambda w}{1+\sqrt{1-4i\overline{\lambda}\left(z_1^2+z_2^2-i\lambda w^2\right)}}.
\end{equation}
\begin{equation}\label{f1case22}
 f_1(z,w)=\dfrac{2z_1}{1+\sqrt{1-4i\overline{\lambda}\left(z_1^2+z_2^2-i\lambda w^2\right)}}.
\end{equation}
\begin{equation}\label{f2case22}
 f_2(z,w)=\dfrac{2z_2}{1+\sqrt{1-4i\overline{\lambda}\left(z_1^2+z_2^2-i\lambda w^2\right)}}.
\end{equation}
Theorem \ref{claim317} is proved.
\end{proof} 
\section{On proper holomorphic maps from $\mathbb{B}^{n}_l$ to $D^{\mathrm{IV}}_{n+1,l}$}
\label{sect:4}

In the case $m=n+1$, by using Theorem \ref{thrmHn1Xn2}, we have explicit formulas
for local proper holomorphic maps from the generalized ball \(\mathbb{B}_1^3\)
into a generalized Lie ball \(D^{\mathrm{IV}}_{3,1}\). The following are
(local) proper
holomorphic maps from the generalized ball $\partial\mathbb{B}^{n}_l$ to the 
Siegel upper half-space $\Omega^{+}_{n,l}$:

\begin{align}\label{mapBtoH1}
    \Upsilon_1(z,w) &= \left(\dfrac{\sqrt{2}z_1}{1+w},\dfrac{\sqrt{2}z_2}{1+w},\dfrac{2i(1-w)}{1+w}\right),\\
    \label{mapBtoH2}
    \Upsilon_2(z,w) &= \left(-\dfrac{\sqrt{2}z_1}{1+w},\dfrac{\sqrt{2}z_2}{1+w},\dfrac{2i(1-w)}{1+w}\right),\\
    \label{mapBtoH3}
    \Upsilon_3(z,w) &= \left(\dfrac{\sqrt{2}z_1}{1+w},-\dfrac{\sqrt{2}z_2}{1+w},\dfrac{2i(1-w)}{1+w}\right),\\
    \label{mapBtoH4}
    \Upsilon_4(z,w) &= \left(-\dfrac{z_1}{1-w},\dfrac{z_2}{1-w},\dfrac{i(1+w)}{1-w}\right),\\
    \label{mapBtoH5}
    \Upsilon_5(z,w) &= \left(\dfrac{z_1}{1+w},\dfrac{z_2}{1+w},\dfrac{i(1-w)}{1+w}\right).
\end{align}
On the other hand, the following is a local biholomorphic map sending a piece of 
$\mathcal{X}^{n+1}_l$ into $D^{\mathrm{IV}}_{n+1,l}$:
\begin{equation}\label{mapXtoD}
    \Omega(z,w)
    =\left(\frac{2iz}{2i+w}, \frac{i-\dfrac{w}{2}-i\zeta-\dfrac{1}{2}\left(w\zeta+izz^t\right)}{2i+w},
    \frac{-1-\dfrac{iw}{2}-\zeta+\dfrac{i}{2}\left(w\zeta+izz^t\right)}{2i+w}\right),
\end{equation}
where $zz^t=z_1^2 + z_2^2 + \cdots + z_{n-1}^2$ for short. 
Let $X(z,\zeta,w)=\{2z,\zeta,4w\}$, which is an automorphism of $\mathcal{X}^7_1$.

For \(n\geq 3\), similarly to the pseudoconvex case \cite{XiaoYuan2020},
we can verify that \(R_0^{(n)}\) and \(I^{(n)}\) are two local proper holomorphic
maps sending a small one-sided neighborhood \(\mathbb{B}_l^n\) into \(D^{\mathrm{IV}}_{n+1,l}\).
When \(n\geq 4\), these are two representatives of the two equivalence classes.

In the special case $n=3$. The additional rational maps can be constructed
by composing with the maps above. Precisely, we have
the following maps as in Corollary \ref{thrmBnlD44l}:
\begin{align*}
  R_0(z,w) & = \Omega\circ H_1 \circ \Upsilon_1, \\
  R_1(z,w) & = \Omega\circ H_2\circ \Upsilon_2, \\
  R_2(z,w) & = \Omega\circ H_3\circ \Upsilon_3, \\
  P_1(z,w) & = \Omega\circ X\circ H_4\circ \Upsilon_4, \\
  P_2(z,w) & = \Omega\circ X\circ H_5\circ \Upsilon_5,
\end{align*}
\begin{remark}\rm
  Each of the rational maps \(R_{1,2,3}\) has 
  the same indeterminacy set
  \[\{(z, w) \in \mathbb{C}^3 \mid w+1 = 0, z_1^2 + z_2^2 = 0\},\]
  which is contained in the boundary \(\partial \mathbb{B}^3_{1}\) of the generalized ball. 
  On the other hand, the pole set \(\{(z, w) \in \mathbb{C}^3 \mid w + 1 = 0\}\)
  meets both sides of the boundary. This exhibits a difference
  from the signature zero case in which \(R_0\) does not have a singularity
  in the unit ball \(\mathbb{B}^3\).
  
  On the open subset 
  \(U_j \subset \partial \mathbb{B}^3_1 \setminus \{w+1 = 0\}\) 
  whose points are mapped
  by \(R_j\) to a smooth point of \(\partial D^{\mathrm{IV}}_{4,1}\), \(j=1, 2, 3\),
  the geometric
  rank of \(R_j\) is constant: \(R_0\) has vanishing geometric rank while 
  \(R_{1,2}\) both have geometric rank 1. This can be verified by direct
  but tedious calculations.
  
  The polynomial maps \(P_{1,2}\) both have geometric rank 2 at points which 
  are mapped to a smooth point of the target.
\end{remark}

We end this section by computing the CR Ahlfors tensor of a map. Let's take \(P_2\), for 
example. We will compute its components in the CR frame
\[
  Z_{j} = \diffp{\rho}{w} \diffp{}{z_j} - \diffp{\rho}{z_j} \diffp{}{w},
  \quad 
  j = 1, 2,
\]
with \(\rho = 1 - |w|^2 + |z_1|^2 - |z_2|^2\) being the ``standard'' defining
function of the generalized ball. Thus, we have
\[
Z_{1} = -\bar{w} \diffp{}{z_1} - \bar{z}_1 \diffp{}{w},
\quad
Z_{2} = -\bar{w} \diffp{}{z_2} + \bar{z}_2\diffp{}{w}.
\]
By direct calculations, we find that
\[
\rho_{D^{\mathrm{IV}}_{3,1}} \circ P_2 = (1 + |w|^2 + |z_1|^2 + |z_2|^2) \rho(z, w).
\]
Hence, we put \(Q(z, w)=1 + |w|^2 + |z_1|^2 + |z_2|^2\) and \(v(z, w) = \log (Q)\).
The complex Hessian \(v_{Z\bar{Z}}\) of \(v\), where \(Z = (z, w)\), can be computed
exactly as in 
the computation of the Fubini-Study metric in the usual affine coordinate patch of 
the complex projective space \(\mathbb{C}P^3\). Precisely,
\[
v_{Z\bar{Z}} = \frac{1}{Q^2} 
\begin{bmatrix}
1+|z_2|^2+|w|^2 & - \bar{z}_1 z_2 & -\bar{z}_1 w\\
- \bar{z}_2 z_1 & 1 + |w|^2 + |z_1|^2 & - \bar{z}_2 w\\
-\bar{w} z_1 & -\bar{w} z_2 & 1+ |z_1|^2+|z_2|^2
\end{bmatrix}.
\]
We immediately see that \(P_2\) has 
geometric rank 2. Moreover, \(v_{Z\bar{Z}}\) is positive definite everywhere and so is
the CR Ahlfors tensor of \(P_2\) at all points which are mapped to a smooth points of the target.

Restricting this to the tangential CR vectors, we obtain the CR
Ahlfors tensor. In terms of the frame above, its components are
given by the following Hermitian matrix (actually, the matrix has real entries)
\begin{equation}\label{e:ahlforsp2}
\mathcal{A}(P_2)_{j\bar{k}}=
\frac{1}{Q^2}\begin{bmatrix}
(|z_1|^2+|w|^2)(1+|z_2|^2) & |z_1|^2|z_2|^2 + |w|^2 + 2|w|^2|z_2|^2 \\
|z_1|^2|z_2|^2 + |w|^2 + 2|w|^2|z_2|^2 & (|w|^2+|z_2|^2)(1+|z_1|^2) + 4|w|^2|z_2|^2
\end{bmatrix}
\end{equation}
We should note that the formula
on the right hand side of \eqref{e:ahlforsp2} is
only meaningful when being restricted to \(\partial\mathbb{B}^3_{1}\)
on which we have a relation between \(|w|^2, |z_1|^2,\) and \(|z_2|^2\).

The computation of \(\mathcal{A}(P_1)\) is almost the same. We start with
\[
\rho_{D^{\mathrm{IV}}_{3,1}} \circ P_2 = (1 + |w|^2 - |z_1|^2 - |z_2|^2) \rho(z, w).
\]
Thus, the CR Ahlfors tensor \(\mathcal{A}(P_1)\) is the restriction of the 
complex Hessian of \(\log(1 + |w|^2 - |z_1|^2 - |z_2|^2)\), which is well-defined
on a suitable open set. Although the rest of the compuatation is very similar,
there is a difference: At every point outside a singular set of \(P_1\) which
is mapped to a smooth point of the target, the CR Ahlfors tensor of \(P_1\) is
nondegerate, but not positive. This show that \(P_1\) and \(P_2\) are not equivalent.

\section{Higher codimensional case}
\label{sect:5}
In this section, we briefly discuss the case of higher but low codimension. In this case,
one expects that under some conditions on the dimensions and signature, 
CR maps between hyperquadrics and the tube exhibit rigidity property. In fact, based on 
recent research on the rigidity of CR maps between spheres and hyperquadrics 
of Huang--Lu--Tang--Xiao \cite{HLTX2020} and Xiao \cite{xiao23}, one can obtain a 
rigidity result for the case of CR maps from a sphere or a hyperquadric into the tube over the 
symmetric form of higher dimension and codimension.
For the sake of completeness, we present two theorems below.
\begin{theorem}
\label{thm:hcodim}
Let \(m\geq n \geq 4\), \(1 \leq l \leq l'\), and \(l \leq (n-1)/2\).
Assume that \(H\) is a smooth CR map from an connected open subset
of \(\mathbb{H}^{2n-1}_l\) into \(\mathcal{X}_{l'}^{2m+1}\). Then \(l\leq \min(l', m-l')\).
Moreover, assume that one of the 
following conditions holds
\begin{enumerate}
\item 
\(l' < \min(2l-1, n-2)\),

\item
\(l' < 2l - 1\) and \(m-l' < n-1\),

\item 
\(m - l' < 2(n-l-1)\) and \(l' < n-2\),

\item
\(m - l' < 2(n-l-1)\) and \(m - l' < n-1\).
\end{enumerate}
Then \(H\) extends to a local holomorphic isometry of the indefinite ``canonical'' K\"ahler
metrics of one-sided neighborhoods of the source and target.
\end{theorem}
The theorem above applies when, for example, \(l = l'=1\) and \(m \leq 2n-4\).
It seems that these ranges of the dimensions and signatures are not optimal.
\begin{proof} 
Consider the holomorphic map \(\Psi \colon \mathbb{C}^{m+1} \to \mathbb{C}^{m+2}\) given
by
\[
\Psi(z,\zeta, w)
=
\left(z_1, \dots, z_{l'}, \frac{1}{2}\left(w \zeta + i zz^t - i \zeta\right),
z_{l'+1},\dots, z_m, \frac{1}{2}\left(w \zeta + i zz^t + i \zeta\right), w \right),
\]
where \((z, \zeta, w) = (z_1, \dots, z_{m-1}, \zeta, w)\) and \(zz^t = z_1^2 + \cdots + z_{m-1}^2\).
Then \(\Psi\) is transversal to \(\mathbb{H}^{2m+3}_{l'+1}\) and sends
\(\mathcal{X}_{l'}^{2m+1}\) into \(\mathbb{H}^{2m+3}_{l'+1}\), where
\[
\mathbb{H}^{2m+3}_{l'+1} 
= \left\{(z_1, \dots, z_{m+1}, w) \in \mathbb{C}^{m+2} \mid \tilde{\rho} := \Im(w) + \sum_{j=1}^{l'+1} |z_j|^2 - \sum_{j=l'+2}^{m+1} |z_k|^2 = 0\right\},
\]
is the real hyperquadric of signature \(l'+1\) in $\mathbb{C}^{m+2}$.

If \(H \colon \mathbb{H}_l^{2n - 1} \to \mathcal{X}_{l'}^{2m+1}\) is a CR transversal
map, then \(\tilde{H} := \Psi \circ H\) is a CR transversal map from \(\mathbb{H}_l^{2n-1}\)
to \(\mathbb{H}_{l'+1}^{2m+3}\). Therefore, by the result of Huang-Lu-Tang-Xiao mentioned
above, \(\tilde{H}\) extends to an isometry of the indefinite complex hyperbolic metrics
of one-sided neighborhoods of the hyperquadrics. On the other hand, \(H\) itself
extends holomorphically to a neighborhood of \(p\) in \(\mathbb{C}^{n}\) by the
well-known Lewy extension theorem. Finally, the isometry of \(\tilde{H}\) implies the 
isometry of \(H\), as desired.
\end{proof}

In the case \(l = 0\) and \(4\leq n \leq m \leq 2n-3\), we can use a result of
Xiao \cite{xiao23} to obtain a rigidity of CR maps from the sphere. When \(l'=0\), the
theorem below is just a result of Xiao-Yuan \cite{XiaoYuan2020} (for \(m\leq 2n-4\)) 
and Xiao \cite{xiao23} (for \(m = 2n-3\)). The proof for the case \(l' > 0\) 
follows the same strategy so it is omitted.
\begin{theorem}
\label{thm:hcodim2}
Let \(m\geq n \geq 4\) and \(0\leq m < 2n-3\).
Assume that \(H\) is a smooth CR transversal map from a connected open subset
of \(\mathbb{H}^{2n-1}_l\) to \(\mathcal{X}_{l'}^{2m+1}\).
Then \(H\) extends to a local holomorphic isometry of the ``canonical'' pseudo-K\"ahler
metric of a one-sided neighborhood of \(\mathbb{H}^{2n-1}_l\) into
the ``canonical'' pseudo-K\"ahler metric of a one-sided neighborhood
of \(\mathcal{X}_{l'}^{2m+1}\).
\end{theorem}
A similar statement can be made for the boundaries of the generalized ball and generalized Lie ball.
However it is not known at the moment if a local isometry in the generalized setting extends to a global holomorphic map defined on the whole generalized ball.

\bibliographystyle{plain}   
\bibliography{refs} 

\end{document}